\documentclass[12pt]{amsart}
\usepackage{amssymb,latexsym}
\usepackage{enumerate} 

\newtheorem{thm}{Theorem}[section]

\newtheorem{lem}[thm]{Lemma}

\numberwithin{equation}{section}

\usepackage{fullpage}

\usepackage[english]{babel}
\usepackage[Symbol]{upgreek}    
\usepackage{bm}
\usepackage{url}

\newcommand{\Odip}[2]{\mathcal{O}_{#1}\!\left(#2\right)\mathchoice{\!}{}{}{}}
\newcommand{\Odi}[1]{\Odip{}{#1}}

\newcommand{\odip}[2]{{o}_{#1}\!\left(#2\right)\mathchoice{\!}{}{}{}}
\newcommand{\odi}[1]{\odip{}{#1}}

\newcommand{\NN}{\mathfrak{N}}

\newcommand{\M}{\mathfrak{M}}
\newcommand{\m}{\mathfrak{m}}
\renewcommand{\t}{\mathfrak{t}}

\newcommand{\Z}{\mathbb{Z}}
\newcommand{\Q}{\mathbb{Q}}
\newcommand{\R}{\mathbb{R}}
\newcommand{\N}{\mathbb{N}}

\newcommand{\singseries}{\mathfrak{S}} 
\newcommand{\second}{\prime\prime}
\newcommand{\Eulerphi}{\varphi}
\newcommand{\dx}{\mathrm{d}}

\newtheoremstyle{Nonumtheorems}
{10pt}
{6pt}
{\itshape}
{}
{\bfseries}
{.}
{.5em}
{\thmname{#1}\thmnote{ (#3)}}

\theoremstyle{Nonumtheorems}
\newtheorem{Nonumthm}{Theorem}
\newtheorem{Nonumcor}{Corollary} 

\allowdisplaybreaks

\begin{document} 

\baselineskip=17pt

\title[On a Diophantine problem]{On a Diophantine problem with one prime,\\ two squares  of primes  and $s$ powers of two}
\author{Alessandro LANGUASCO}   
\address{Dipartimento di Matematica Pura e Applicata, Universit\`a
di Padova, Via Trieste 63, 35121 Padova, Italy} \email{languasco@math.unipd.it}

\author{Valentina SETTIMI} 
  \address{Dipartimento di Matematica Pura e Applicata, Universit\`a
di Padova, Via Trieste 63, 35121 Padova, Italy} \email{vsettimi@math.unipd.it}
\date{}
%
%
\begin{abstract}
We refine a result of W.P.~Li and  Wang \cite{LiW2005}
on the values of the form
$
\lambda_1p_1
+
\lambda_2p_2^{2}
+
\lambda_3p_3^{2}
+
\mu_1 2^{m_1}
+
\dotsm
+
\mu_s 2^{m_s},
$
where $p_1,p_2,p_3$ are prime numbers, $m_1,\dotsc, m_s$ are positive
integers, $\lambda_1,\lambda_2,\lambda_{3}$ are nonzero real numbers,
not all of the same sign,
$\lambda_2 / \lambda_3$ is  irrational and
$\lambda_i/\mu_i \in \Q$, for $i\in\{1,2,3\}$.
\end{abstract}
\maketitle
\section{Introduction}
In this paper we are  interested to study the values of the form
\begin{equation}
\label{linear-form}
\lambda_1p_1
+
\lambda_2p_2^{2}
+
\lambda_3p_3^{2}
+
\mu_1 2^{m_1}
+
\dotsm
+
\mu_s 2^{m_s},
\end{equation}
where $p_1,p_2,p_{3}$ are prime numbers, $m_1,\dotsc, m_s$
are positive integers, and the coefficients $\lambda_1$, $\lambda_2$, $\lambda_3$ and
$\mu_1$, \dots, $\mu_s$ are real numbers satisfying suitable relations.

This problem can be seen as a variation of the Waring-Goldbach and the  
Goldbach-Linnik problems. A huge literature is present for both problems
and so we will mention just some of the most important results.

 Concerning the Goldbach-Linnik problem, the first result was proved by Linnik himself
\cite{Linnik1951,Linnik1953} who proved that every sufficiently large even integer is a sum of  two primes and  a suitable number $s$ of powers of two; 
he gave no explicit estimate of $s$. Other results
were proved by Gallagher \cite{Gallagher1975},
J.~Liu-M.-C.~Liu-Wang \cite{LiuLW1998,LiuLW1998a, LiuLW1999},
Wang \cite{Wang1999} and H.~Li \cite{Li2000, Li2001}.
Now the best conditional result is due to Pintz-Ruzsa 
\cite{PintzR2003} and Heath-Brown-Puchta \cite{Heath-BrownP2002}
($s=7$ suffices under the assumption of the Generalized Riemann Hypothesis),
while, unconditionally, it is due to Heath-Brown-Puchta
\cite{Heath-BrownP2002} ($s = 13$ suffices).
Elsholtz, in unpublished work, improved it to $s=12$.
We should also remark that Pintz-Ruzsa announced a proof for the
case $s=8$ in their paper \cite{PintzR2003}.  
Looking for the size of the exceptional set of the Goldbach problem
we recall the fundamental paper
by Montgomery-Vaughan \cite{MontgomeryV1975}
in which they showed that the number of even integers up to $X$ that
are not the sum of two primes is $\ll X^{1-\delta}$.
Pintz \cite{Pintz2006} announced that $\delta=1/3$ is admissible
in the previous estimate. Concerning the
exceptional set for the  Goldbach-Linnik problem, 
Languasco-Pintz-Zaccagnini \cite{LanguascoPZ2007}
proved that for every $s \ge 1$, there
are $\ll X^{3/5} (\log X)^{10}$ even integers in $[1,X]$ that are not
the sum of two primes and $s$ powers of two.

In diophantine approximation several results were proved
concerning  linear forms with primes that, in some
sense, can be considered as the real analogues
of the binary and ternary Goldbach problems.
On this topic we recall the papers by
Vaughan \cite{Vaughan1974,Vaughan1974a,Vaughan1976},
Harman \cite{Harman1991},
Br\"udern-Cook-Perelli \cite{BrudernCP1997},
and Cook-Harman \cite{CookH2006}.
A diophantine problem with two primes and powers of two was solved by Parsell
\cite{Parsell2003}; his estimate on the needed powers of two was recently improved by Languasco-Zaccagnini \cite{LanguascoZ2010b}.

The problem of representing an integer using a suitable 
number of prime powers is usually called the Waring-Goldbach 
problem. We refer to the beautiful Vaughan-Wooley survey paper \cite{VaughanW2002}
for the literature on this problem. Here we just mention that
in 1938 Hua \cite{Hua1938}  proved that almost all the integers $n\equiv 3 \bmod{24}$ and $n\not \equiv 0 \bmod{5}$  are representable as sums of three squares of primes, and all sufficiently large $n\equiv 5\bmod{24}$  are representable as sums of five squares of primes. Also several results were obtained about the size of the exceptional
set for this problem. On this topic we just recall a recent result of J.~Liu, Wooley and Yu
\cite{LiuWY2004}.

Concerning mixed problems with powers of primes and powers of two, we recall the results  by H.~Li \cite{Li2006}, \cite{Li2007},
J.~Liu and L{\"u} \cite{LiuL2004},
J.~Liu and M.-C.~Liu \cite{LiuL2000},
L{\"u} and Sun \cite{LuS2009},
Z.~Liu and L{\"u} \cite{LiuL2010}.

Replacing one of the prime summands in the problem in Parsell \cite{Parsell2003}
with the sum of two squares of primes,
we obtain the problem in \eqref{linear-form};
the only result we know about it is by 
W.P.~Li and  Wang \cite{LiW2005}. 
We improve their estimate on $s$ with the following result whose
quality  depends on rational approximations to
$\lambda_2 / \lambda_3$. 
\begin{Nonumthm}
Suppose that $\lambda_1<0$, $\lambda_2,\lambda_3>0$
such that $\lambda_2/\lambda_3$ is irrational. 
Further suppose that $\mu_1, \dotsc,  \mu_s$ are nonzero
real numbers such that 
$\lambda_i/\mu_i \in \Q$, for $i\in\{1,2,3\}$, and denote by
$a_i/q_i$  their reduced representations as rational numbers.
Let moreover $\eta$ be a sufficiently small positive constant such that
$\eta<\min(\vert\lambda_1/a_1 \vert; \lambda_2/a_2; \lambda_3/a_3)$.
Finally let
\begin{equation}
\label{s0-def}
s_0
=
3
+
\Bigl\lceil
\frac
{
\log 
\left(
4C(q_1,q_2,q_{3},\epsilon)
(\vert \lambda_1 \vert + \vert \lambda_2 \vert  + \vert \lambda_3\vert)
\right)
-
\log ((3-2\sqrt{2}-\epsilon) \eta)
}
{-\log (0.8844472132)} 
\Bigr\rceil,
\end{equation}
where $\epsilon > 0$ is an arbitrarily small constant, 
$C(q_1,q_2,q_{3},\epsilon) $ 
satisfies
\begin{align}
\notag
C(q_1,q_2,q_{3},\epsilon) 
&=
(1+\epsilon)
\Bigl(
\log{2}
+
C \cdot  
\singseries^{\prime}(q_1) 
\Bigr)^{1/2}
\\
\label{Cq1q2q3-def}  
&
\qquad
\times
\Bigl(
\log^{2}{2}
+
D \cdot  
\singseries^{\second}(q_2) 
\Bigr)^{1/4}
\Bigl(
\log^{2}{2}
+
D \cdot  
\singseries^{\second}(q_3) 
\Bigr)^{1/4}, 
\end{align}
$C = 10.0219168340$,
$D = 17, 646, 979.6536361512,$
%
\begin{equation}
\label{singseries-primesecond-def}
  \singseries^{\prime}(n)
  =
  \prod_{\substack{p \mid n \\ p > 2}} \frac{p - 1}{p - 2}
  \quad
  \textrm{and}
  \quad
  \singseries^{\second}(n)
=
\prod_{\substack{p \mid n \\ p > 2}}
\frac{p+1}{p}  .
\end{equation}
Then  for every real number $\varpi$ and every integer
$s\geq s_0$ the  inequality
\begin{equation}
\label{main-inequality}
\vert
\
\lambda_1p_1
+
\lambda_2p_2^{2}
+
\lambda_3p_3^{2}
+
\mu_1 2^{m_1}
+
\dotsm
+
\mu_s 2^{m_s}
+
\varpi
\
\vert
<
\eta
\end{equation}
has infinitely many solutions in primes
 $p_1,p_2,p_{3}$ and positive integers $m_1,$ $\dotsc,$ $m_s$.
\end{Nonumthm}

Arguing analogously we can prove the case
$\lambda_{1},\lambda_{2}<0$,  $\lambda_{3}>0$, see the
argument at the bottom of \S \ref{major-arc-eval}.

Our value in \eqref{s0-def} largely improves W.P.~Li-Wang's \cite{LiW2005} 
one given by
\begin{equation}
\label{Li-Wang-1}
s_0
=
3+
\Bigl\lceil
\frac
{
\log 
\left(
2^{9}C_1(q_1,q_2,q_3,\epsilon)
(\vert \lambda_1 \vert + \vert \lambda_2\vert+ \vert \lambda_3\vert)^{2}
\right)
- 
\log  ((1-\epsilon)\vert \lambda_{1} \vert \eta)
}
{-\log (0.995)}
\Bigr\rceil,
\end{equation} 
where 
\begin{align}
\label{Li-Wang-2}
C_1(q_1,q_2,q_3,\epsilon)
=
&
5(1+\epsilon)
\Bigl(
\frac{11^{4}\cdot43\cdot \pi^{26}}{2^{27}\cdot 25}+\log^{2} 2
\Bigr)^{1/2}  
\\
\notag
&
\times
(\log 2q_1)^{1/2}(\log 2q_2)^{1/4}(\log 2q_3)^{1/4}.
\end{align}
Comparing only denominators in \eqref{s0-def} 
and in  \eqref{Li-Wang-1}, we see
that our gain is about $95.9$\%.
Moreover the numerical constants involved in the definition \eqref{Cq1q2q3-def}
are better than the ones in  \eqref{Li-Wang-2}, see the remark after
Lemma \ref{sqroot-Dioph-equation} below.

In practice, the following example shows that the gain is actually
slightly larger.
For instance, taking $\lambda_1=  -\sqrt{5}= \mu_1^{-1}$,
$\lambda_2= \sqrt{3}= \mu_2^{-1}$, 
$\lambda_3= \sqrt{2}= \mu_3^{-1}$, $\eta=1$ and $\epsilon=10^{-20}$, 
we get $s_0= 120$.
W.P.~Li-Wang's estimate \eqref{Li-Wang-1} 
gives $s_0= 4120$.

Moreover we remark that the  works of Rosser-Schoenfeld 
\cite{RosserS1962} on $n/\Eulerphi(n)$ and of Sol\'e-Planat \cite{SoleP2011}
on the Dedekind $\Psi$ function, see Lemmas \ref{sing-series-estim} and \ref{sing-series-estim-two} below,
give for $\singseries'(q)$ and $\singseries^{\second}(q) $ a sharper estimate than $2\log (2q)$, 
used in \eqref{Li-Wang-2}, for large values of $q$.

With respect to \cite{LiW2005}, 
our main gain comes 
from enlarging the size of the major arc since this lets us
use sharper estimates on the minor arc.
In particular, on the major arc we replaced the technique used in 
\cite{LiW2005} with an  argument involving  an $L^{2}$-estimate of the exponential sum over prime squares  ($S_2(\alpha)$). This is a standard tool when working on primes
(see, \emph{e.g.}, \cite{LanguascoZ2010b} for an application
to a similar problem) but
it seems that it is the first time that a similar technique is
used for prime squares so we inserted a detailed proof of the
relevant lemmas (Lemmas \ref{sqroot-BCP-Gallagher} and 
\ref{sqroot-Saffari-Vaughan} below) since they could be of some independent interest.

On the minor arc we used Ghosh estimate  \cite{Ghosh1981}  to deal 
with the exponential sum on primes squares  while 
to treat the exponential sum on primes ($S_1(\alpha)$) we follow 
the argument in  \cite{LanguascoZ2010b}.
To work with the exponential sum over powers of two
($G(\alpha)$),  we inserted Pintz-Ruzsa's \cite{PintzR2003} algorithm to estimate 
the measure of the subset of the minor arc on which $\vert G(\alpha)\vert$ is ``large''.
These ingredients lead to
a sharper estimate on the minor arc 
and let us improve the size of the denominators in 
\eqref{s0-def}.

A second, less important, gain arises from our Lemmas \ref{Dioph-equation}
and \ref{sqroot-Dioph-equation} below
which improves the numerical values
in \eqref{Cq1q2q3-def} comparing with
the ones in  \eqref{Li-Wang-2} (see also Parsell \cite{Parsell2003}, Lemma 3).

Using the notation
$\bm{\lambda}=(\lambda_1,\lambda_2, \lambda_3)$,
$\bm{\mu}=(\mu_1,\mu_2, \mu_{3})$, as a consequence of the Theorem we have the 
\begin{Nonumcor} 
Suppose that $\lambda_1,\lambda_2,\lambda_{3}$ are nonzero real numbers,
not all of the same sign,
such that $\lambda_2/\lambda_3$ is irrational.
Further suppose $\mu_1, \dotsc,  \mu_s$ are nonzero
real numbers such that $\lambda_i/\mu_i \in \Q$, for $i\in\{1,2,3\}$, 
and denote by $a_i/q_i$  their reduced representations as rational numbers.
Let moreover $\eta$ be a sufficiently small positive constant such that
$\eta<\min(\vert \lambda_1/a_1\vert ;
\vert \lambda_2/a_2 \vert;$ $\vert \lambda_3/a_3 \vert)$ and 
$\tau\geq\eta>0$.
Finally let $s_0=s_0(\bm{\lambda},\bm{\mu},\eta, \epsilon)$ 
as defined  in \eqref{s0-def}, where $\epsilon>0$ is arbitrarily small.
Then for every real number $\varpi $ and every integer
$s\geq s_0$ the  inequality
\begin{equation*} 
\vert
\
\lambda_1p_1
+
\lambda_2p_2^{2}
+
\lambda_3p_3^{2}
+
\mu_1 2^{m_1}
+
\dotsm
+
\mu_s 2^{m_s}
+
\varpi
\
\vert
<
\tau
\end{equation*}
has infinitely many solutions in primes
 $p_1,p_2,p_{3}$ and positive integers $m_1,$ $\dotsc,$ $m_s$.
\end{Nonumcor}

This Corollary immediately follows from the Theorem 
by rearranging the  $\lambda$'s.
Hence the Theorem assures us that \eqref{main-inequality}
has infinitely many solutions and the Corollary immediately 
follows from the condition $\tau\geq\eta$.

\section{Definitions}

Let $\epsilon$ be a sufficiently small positive constant
(not necessarily the same at each occurrence), $X$ be a large parameter, 
$M=\vert \mu_1 \vert + \dotsm +\vert \mu_s \vert$
and $L=\log_2 (\epsilon X/(2M))$, where $\log_2 v$ is
the base $2$ logarithm of $v$.
We will use the Davenport-Heilbronn  variation of the Hardy-Littlewood
method to count the number $\NN(X)$ of solutions of the 
inequality \eqref{main-inequality}  with
$\epsilon X \leq p_1, p_2^{2}, p_3^{2} \leq X$
and $1\leq m_1, \dotsc, m_s \leq L$.
Let now $e(u) = \exp(2\pi i u)$ and
\[
S_{1}(\alpha) = \sum_{\epsilon X\leq p \leq X} \log p\ e(p\alpha),
\quad
S_{2}(\alpha) = \sum_{\epsilon X\leq p^{2} \leq X} \log p\ e(p^{2}\alpha),
\]
\[
G(\alpha)= \sum_{1 \leq m \leq L}  e(2^m\alpha).
\]
For $\alpha\neq 0$, we also define
\begin{equation}
\notag
K(\alpha,\eta)=
\Bigl(\frac{\sin\pi \eta \alpha}{\pi \alpha}\Bigr)^2
\end{equation}
and hence, denoting the real line by $\R$, both 
\begin{equation}
\label{K-transform}
\widehat{K}(t,\eta)
=
\int_\R K(\alpha,\eta) e(t\alpha ) \dx\alpha
=
\max(0; \eta -\vert t\vert)
\end{equation}
and 
\begin{equation}
\label{K-inequality}
K(\alpha,\eta)
\ll
\min(\eta^2; \alpha^{-2})
\end{equation}
are well-known facts. 
Letting
\begin{equation}
\notag
I(X ; \R)
=
\int_{\R}
S_{1}(\lambda_1 \alpha) S_{2}(\lambda_2 \alpha)
S_{2}(\lambda_3 \alpha)
G(\mu_1 \alpha)\dotsm G(\mu_s \alpha)
e(\varpi\, \alpha)
K(\alpha,\eta)
\dx \alpha,
\end{equation}
it follows from \eqref{K-transform} that
\[
I(X ; \R)
  \ll
\eta \log^3 X \cdot \NN(X).
\]
We will prove, for $X\to +\infty$ running over a suitable integral sequence,  that
\begin{equation}
\label{I-lower-bound}
I(X ; \R)
\gg_{s,\bm{\lambda},\epsilon}
\eta^2 X (\log X)^s
\end{equation}
thus obtaining
\[
 \NN(X) \gg_{s,\bm{\lambda},\epsilon} \eta  X (\log X)^{s-3}
\]
and hence the Theorem follows.

To prove \eqref{I-lower-bound}
we first dissect the real line in the major, minor and trivial arcs,
by choosing  $P=X^{2/5}/\log X$  and letting
\begin{equation}
\label{dissect-def}
\M = \{\alpha\in \R: \vert \alpha \vert \leq P/X \}, 
\quad 
\m = \{\alpha\in \R: P/X<\vert \alpha \vert \leq L^2\},
\end{equation}
and $\t = \R \setminus(\M\cup\m)$.
Accordingly, we  write
\begin{equation}
\label{integral-dissect}
I(X ; \R)
=
I(X ; \M) + I(X ; \m) + I(X ; \t).
\end{equation}
We will prove that 
\begin{equation}
\label{major-goal}
I(X ; \M)
\geq
c_1  \eta^2 X L^s,
\end{equation}
\begin{equation}
\label{trivial-goal}
\vert 
I(X ; \t)
\vert
=
\odi{XL^s}
\end{equation}
both hold for all sufficiently large $X$, and
\begin{equation}
\label{minor-goal}
\vert
I(X ; \m)
\vert
\leq
c_2(s) \eta X L^s
\end{equation}
holds for $X\to +\infty$ running over a suitable integral sequence, 
where $c_2(s)>0$ depends on $s$,
$c_2(s)\to 0$ as $s\to+\infty$,  and
$c_1=c_1(\epsilon,\bm{\lambda})>0$ is a constant such that
\begin{equation}
\label{constant-condition}
c_1 \eta -c_2(s) \geq c_3\eta
\end{equation}
for some absolute positive constant  $c_3$ and $s\geq s_0$.
Inserting \eqref{major-goal}-\eqref{constant-condition} into \eqref{integral-dissect},
we finally obtain that \eqref{I-lower-bound} holds thus proving the Theorem.

\section{Lemmas}
Let $n$ be a positive integer.
We denote by $\singseries(n)$ the singular series and set
$\singseries(n) = 2 c_0 \singseries^{\prime}(n)$ where $\singseries^{\prime}(n)$
is defined in \eqref{singseries-primesecond-def} and 
\begin{equation}
\notag 
   c_0
  =
   \prod_{p > 2} \Big( 1 - \frac{1}{(p-1)^2} \Big).
\end{equation}
Notice that $\singseries^{\prime}(n)$ is a multiplicative function.
According to Gourdon-Sebah \cite{GourdonS2001},
we can also write that
$0.66016181584<c_0<0.66016181585$.
 
The first lemma is an upper bound for the multiplicative part of
the singular series.
\begin{lem}[Languasco-Zaccagnini \cite{LanguascoZ2010b}, Lemma 2]
\label{sing-series-estim} 
For $n\in \N$, $n\geq 3$, we have that
\[
\singseries^{\prime}(n)
<
\frac{n}{c_0\Eulerphi(n)}
<
\frac{e^{\upgamma} \log \log n}{c_0}
+
\frac{2.50637}{c_0 \cdot \log \log n},
\]
where $\upgamma=0.5772156649\dotsc$ is the Euler constant.
\end{lem}
Letting $f(1)=f(2)=1$ and 
$ f(n) = n/(c_0\Eulerphi(n)) $
for  $n \geq 3$,
we can say that the inequality $\singseries^{\prime}(n) \leq f(n)$
is sharper than Parsell's estimate  $\singseries^{\prime}(n) \leq 2 \log(2 n)$,
see page 369 of \cite{Parsell2003},  for every $n\geq 1$.
Since it is clear that computing the exact value of $f(n)$ for
large values of $n$  is not easy (it requires the knowledge of
every prime factor of $n$), we also remark that
the second estimate in Lemma \ref{sing-series-estim} leads to
a sharper bound  than $\singseries^{\prime}(n) \leq 2 \log(2 n)$ 
for every $n\geq 14$.

Let now $\singseries^{\second}(n)$ defined as in \eqref{singseries-primesecond-def}.
We first remark that it is connected with the Dedekind $\Psi$ function
defined by
\[
\Psi(n) = n \prod_{\substack{p \mid n}}
\frac{p+1}{p} 
\]
since $\singseries^{\second}(n) = \Psi(n)/n$  for $n$ odd and 
$\singseries^{\second}(n) = (2/3) \Psi(n)/n$  for $n$ even.
We also have
\begin{lem}
\label{sing-series-estim-two} 
For $n\in \N$, $n\geq 31$, we have 
\[
\singseries^{\second}(n) 
<
e^{\upgamma} \log \log n,
\]
where $\upgamma =0.5772156649\dotsc$ is the Euler constant.
\end{lem}
\begin{proof} 
This follows immediately from Corollary 2 of Sol\'e-Planat \cite{SoleP2011}
and the previous remarks.
\end{proof}
The  estimate in Lemma \ref{sing-series-estim-two}
is sharper than W.P.~Li-Wang's one  $\singseries^{\second}(n) \leq 2 \log(2 n)$,
see page 171 of  \cite{LiW2007},  for every $n\geq 31$.
We also remark that $\singseries^{\second}(1)=\singseries^{\second}(2)=1$, 
and that the computation of $\singseries^{\second}(n)$ in the remaining 
interval $3\leq n \leq 30$ is an easy task.

Now we state some lemmas we need to estimate  $I(X ;\m)$. 
\begin{lem} [Languasco-Zaccagnini \cite{LanguascoZ2010b}, Lemma 4] 
\label{Dioph-equation}
Let $X$ be a sufficiently large parameter and let $\lambda, \mu \neq
0$ be two real numbers such that $\lambda / \mu\in \Q$.
Let $a,q\in \Z\setminus\{0\}$ with $q>0$, $(a,q)=1$ be such
that $\lambda/\mu= a/q$.
Let further $0<\eta < \vert \lambda/a \vert $.
We have
\[
\int_{\R}
\vert
S_1(\lambda \alpha) G(\mu \alpha)
\vert^2
K(\alpha, \eta)
\dx\alpha
<
\eta X L^2
\Bigl(
(1-\epsilon)\log{2}
+
C \cdot 
\singseries'(q) 
\Bigr)
+
\Odip{M,\epsilon}{\eta X L},
\]
where 
$C =  10.0219168340$.
%
\end{lem}
\begin{lem}
\label{LiuL-lemma}
Let $\epsilon$ be an arbitrarily small positive constant.  
Letting  $n\in \Z$,  $n\neq 0$,
$\vert n \vert \leq X$,
$n\equiv 0 \bmod {24}$
and
 \begin{align*}
 r(n)
 =
 \bigl\vert
 \bigl\{
 n=p_1^2+p_2^2-p_3^2-p_4^2,
\, \textrm{where}
 \
 p_j\leq X^{1/2}, j=1,\dotsc,4
 \bigr\}
 \bigr\vert.
 \end{align*}
 We have 
 \[
 r(n)
 \leq 
 (1+\epsilon) c_{4} \frac{\pi^2}{16}\singseries_{-}(n)
 \frac{X}{\log^4 X},
 \]
 where  
 \[
 \singseries_{-}(n)
 =
 \left(
 2-\frac{1}{2^{\beta_0-1}}-\frac{1}{2^{\beta_0}}
 \right)
 \prod_{\substack{p>2\\p^\beta \parallel n \\ \beta\geq0}}
 \left(
 1+\frac{1}{p}-\frac{1}{p^{\beta+1}}-\frac{1}{p^{\beta+2}}
 \right),
 \]
 $c_{4}= (101)\cdot2^{20}$
 and 
 $\beta_0$ is such that $2^{\beta_0}\parallel n$.
\end{lem}

Lemma \ref{LiuL-lemma} follows by inserting the remark at page
385 of H.~Li \cite{Li2006} in the proof of Lemma 2.2 of J.~Liu-L\"u \cite{LiuL2004}.
We immediately remark that $\singseries_{-}(n) \leq 2\singseries^{\second}(n)$.

We will also need the following
\begin{lem}[H.~ Li \cite{Li2006}]
\label{HongzeLi-lemma}
Let $d$ be a positive odd integer and 
$\xi(d)$ be the quantity 
$\min\{ \mu \colon 2^\mu$  $\equiv 1 \pmod{d} \}$.
Then the series
\[
 \sum_{\substack{d=1\\ 2\nmid d}}^{+\infty}
   \frac{\mu^2(d)}{d\xi(d)}
\]
is convergent and its value $c_{5}$ satisfies $c_{5}<1.620767$.
\end{lem}

The next lemma is the analogue of Lemma \ref{Dioph-equation}
for exponential sums over prime squares.
\begin{lem}
\label{sqroot-Dioph-equation}
Let $X$ be a sufficiently large parameter and let $\lambda, \mu \neq
0$ be two real numbers such that $\lambda / \mu\in \Q$.
Let $a,q\in \Z\setminus\{0\}$ with $q>0$, $(a,q)=1$ be such
that $\lambda/\mu= a/q$.
Let further $0<\eta < \vert \lambda/a \vert $.
We have
\[
\int_{\R}
\vert
S_2(\lambda \alpha) G(\mu \alpha)
\vert^4
K(\alpha, \eta)
\dx\alpha
<
(1+\epsilon)
\eta   X L^4
\Bigl(
 \log^2 2
+
D\cdot \singseries^{\second}(q)
\Bigr),
\]
where $D = c_{4}c_{5}\pi^2 /(96)$,
$c_{4}, c_{5}$ are respectively defined as in 
Lemmas \ref{LiuL-lemma}-\ref{HongzeLi-lemma}
and $\epsilon$ is an arbitrarily small positive constant.  
\end{lem}
 This should be compared with Lemma 4.3 of W.P.~Li-Wang \cite{LiW2007}
 (see also Lemma 4.2 of \cite{LiW2005}) in which the value 
 $D_{1} =2^{-27} \cdot 11^{4}\cdot43\cdot \pi^{26}/(25)$ plays the role of $D$.
Using the values $c_{4}=101\cdot 2^{20}$ and $c_{5}<1.620767$ as in 
Lemmas \ref{LiuL-lemma}-\ref{HongzeLi-lemma}, we see that 
$D<17, 646, 979.6536361512$ while $D_{1} = 1, 581, 925, 383.0798448770$. 
We remark that $D< (0.0112)\cdot D_1$ and so the
reduction factor here is close to the $98.8$\%.
With an abuse of notation, in the statement of
the Theorem we will set  $D= 17, 646, 979.6536361512$.

\begin{proof}
Letting now
\[
I
=
\int_{\R}
\vert
S_2(\lambda \alpha) G(\mu \alpha)
\vert^4
K(\alpha, \eta)
\dx\alpha,
\]
by \eqref{K-transform} we immediately have
\begin{align}
\notag
I
&=
\sum_{\epsilon X  \leq p_1^2, p_2^2,p_3^2,p_4^2\leq X}
\log p_1 \log p_2 \log p_3 \log p_4
\\
\label{expanded}
&
\times
\!\!\!\!\!\!\!\!\!\!\!\!
\sum_{1 \leq m_1, m_2, m_3,m_4 \leq L}
\!\!\!\!\!\!\!\!\!\!\!\!
\max
\Bigl(
0;
\eta - \vert
\lambda(p_1^2 + p_2^2-p_3^2-p_4^2) 
+
\mu(2^{m_1}+2^{m_2}-2^{m_3}-2^{m_4})
\vert
\Bigr).
\end{align}
Let $\delta= \lambda(p_1^2 + p_2^2-p_3^2-p_4^2) 
+
\mu(2^{m_1}+2^{m_2}-2^{m_3}-2^{m_4})
$.
For a sufficiently small $\eta>0$, we claim that
\begin{equation}
\label{delta0}
\vert
\delta
\vert
< \eta
\quad
\textrm{is equivalent to}
\quad
\delta = 0.
\end{equation}
Recall our hypothesis on $a$ and $q$, and assume that
$\delta\neq 0$ in \eqref{delta0}.
For $\eta<\vert \lambda/a\vert$ this leads to a contradiction. In fact we have
\begin{align*}
\frac{1}{\vert a \vert}
>
\frac{\eta}{\vert \lambda\vert}
&>
\Bigl\vert
p_1^2 + p_2^2-p_3^2-p_4^2 
+ 
\frac{q}{a}
(2^{m_1}+2^{m_2}-2^{m_3}-2^{m_4})
\Bigr\vert
\\
&
=
\Bigl\vert
\frac{a(p_1^2 + p_2^2-p_3^2-p_4^2) + q
(2^{m_1}+2^{m_2}-2^{m_3}-2^{m_4})}{a}
\Bigr\vert
\geq
\frac{1}{\vert a \vert},
\end{align*}
since $a(p_1^2 + p_2^2-p_3^2-p_4^2) + q
(2^{m_1}+2^{m_2}-2^{m_3}-2^{m_4})\neq 0$ is a linear integral
combination.
Inserting \eqref{delta0} in \eqref{expanded},
for $\eta<\vert \lambda/a\vert$ we can write that
\begin{equation}
\label{expanded1}
I
=
\eta
\sum_{\epsilon X  \leq p_1^2, p_2^2,p_3^2,p_4^2\leq X}
\sum _{\substack{
1 \leq m_1, m_2, m_3,m_4 \leq L \\ \\ \hskip-2.9cm
 \lambda(p_1^2 + p_2^2-p_3^2-p_4^2) 
+
\mu(2^{m_1}+2^{m_2}-2^{m_3}-2^{m_4}) 
=0
}}
\log p_1 \log p_2 \log p_3 \log p_4.
\end{equation}
The diagonal contribution  in \eqref{expanded1}
is equal to
\begin{equation}
\label{diagonal-contrib}
\eta
\sum_{\substack
{\epsilon X  \leq p_1^2, p_2^2,p_3^2,p_4^2\leq X
\\
p_1^2 + p_2^2=p_3^2+p_4^2}}
\log p_1 \log p_2 \log p_3 \log p_4
\sum _{\substack{
1 \leq m_1, m_2, m_3,m_4 \leq L \\ 
2^{m_1}+2^{m_2}=2^{m_3}+2^{m_4}
}}
1.
\end{equation}
The number of the solutions of $p_1^2 + p_2^2=p_3^2+p_4^2$
when $p_1p_2\neq p_3p_4$ can be estimated using Satz 3, page 94, of 
Rieger \cite{Rieger1968} and it is
$\ll X (\log X)^{-3}$. This gives a contribution
to the first sum which is $\ll X \log X$.
In the remaining case $p_1p_2 = p_3p_4$ the first sum becomes
\begin{align*}
2
\sum_{\epsilon X  \leq p_1^2, p_2^2\leq X}
\log^2 p_1 \log^2 p_2
&=
2
\Bigl(
\sum_{\sqrt{\epsilon X} \leq p\leq \sqrt{X}}
\log^2 p
\Bigr)^2 
<
(1-\epsilon)
\frac{X}{2}  \log^2 X,
\end{align*}
where we used the Prime Number Theorem
and the fact that $\epsilon$ is a sufficiently small
positive constant.
The sum over the powers of two in \eqref{diagonal-contrib}
can be evaluated by fixing first $m_1 = m_3$ (thus 
getting exactly $L^2$ solutions) and then fixing
 $m_1 \neq m_3$ (which gives other $L^2-L$
 solutions). Hence the contribution of the second sum in
 \eqref{diagonal-contrib} is $2L^2-L$.
 
Combining these results we get that
the total contribution of  \eqref{diagonal-contrib}
is 
\begin{equation}
\label{diag-eval}
<
(1-\epsilon)
\eta
X L^2 \log^2 X
< 
\eta
X L^4 \log^2 2.
 \end{equation}

Now we have to estimate
the contribution $I'$ of the non-diagonal solutions
of $\delta=0$ and we will achieve this by connecting $I'$ with
the singular series of Lemma \ref{LiuL-lemma}.
First, we remark that if $p_{j}>3$  for every $j=1,\dotsc,4$,
then $n=p_1^2+p_2^2-p_3^2-p_4^2 \equiv 0 \bmod{24}$.
So if $n=p_1^2+p_2^2-p_3^2-p_4^2 \not \equiv 0 \bmod{24}$
then at least one of the $p_{j}$ must be equal to $2$ or $3$
and hence  $r(n)$, defined as in the statement
of Lemma \ref{LiuL-lemma}, verifies 
$r(n)\ll X^{1/2+\epsilon}$. 
Recalling that $\lambda/\mu = a/q \neq 0$, $(a,q)=1$,
if $2^{m_3}+2^{m_4}-2^{m_1}-2^{m_2}\neq 0$
and $(q/a)(2^{m_3}+2^{m_4}-2^{m_1}-2^{m_2})\not \equiv 0 \bmod{24}$, we have 
\[
\vert
\{(p_1,\dotsc,p_4) \, : \,
p_1^2 + p_2^2-p_3^2-p_4^2=
(q/a)(2^{m_3}+2^{m_4}-2^{m_1}-2^{m_2})
\}
\vert
\ll
X^{1/2+\epsilon}.
\]

Otherwise,
by Lemma \ref{LiuL-lemma},  $\singseries_{-}(n) \leq 2\singseries^{\second}(n)$, 
$r((q/a)(2^{m_3}+2^{m_4}-2^{m_1}-2^{m_2}))\neq0$
if and only if
$a \mid(2^{m_3}+2^{m_4}-2^{m_1}-2^{m_2})$,
$\log p_j \leq (1/2)\log X$ 
and
$\vert (q/a)(2^{m_3}+2^{m_4}-2^{m_1}-2^{m_2})\vert
\leq 
4\epsilon X/(2M)\vert q/a\vert
\leq
2\epsilon X/\vert \lambda\vert
<
X$
for $\epsilon$ sufficiently small,
we have
\begin{align}
\notag
I' 
& \leq
\frac{\eta}{16} \log ^4 X
\sum_{1\leq m_1, m_2,m_3,m _4 \leq L}
r
\Bigl(
\frac{q}{a}(2^{m_3}+2^{m_4}-2^{m_1}-2^{m_2})
\Bigr)
\\ 
&
\label{Sol-estim1}
<
 (1+\epsilon) 
c_{4} \frac{\pi^2}{128}  
\eta X 
\sum_{1\leq m_1, m_2,m_3,m _4 \leq L}
\singseries^{\second}
\Bigl(
\frac{q}{a}(2^{m_3}+2^{m_4}-2^{m_1}-2^{m_2})
\Bigr).
\end{align}
Using the multiplicativity of
$\singseries^{\second} (n)$ (defined in \eqref{singseries-primesecond-def}),
we get
\begin{align*}
\singseries^{\second}
\Bigl(
\frac{q}{a}(2^{m_3}+2^{m_4}-2^{m_1}-2^{m_2})
\Bigr)
&\leq
\singseries^{\second}(q)
\singseries^{\second}
\Bigl(
\frac{2^{m_3}+2^{m_4}-2^{m_1}-2^{m_2}}{a}
\Bigr)
\\
&
\leq
\singseries^{\second}(q)
\singseries^{\second}(2^{m_3}+2^{m_4}-2^{m_1}-2^{m_2})
\end{align*}
and so, by \eqref{Sol-estim1},
we can write, for every sufficiently large $X$, that
\[
I'
\leq
(1+\epsilon) 
c_{4} \frac{\pi^2}{128}  
\singseries^{\second}(q)
\eta X 
\sum_{1\leq m_1, m_2,m_3,m _4 \leq L}
\singseries^{\second}(2^{m_3}+2^{m_4}-2^{m_1}-2^{m_2}).
\]
Arguing now as at pages 63-64 of J.~Liu-L\"u \cite{LiuL2004}
we have that 
\[
\sum_{1\leq m_1, m_2,m_3,m _4 \leq L}
\singseries^{\second}(2^{m_3}+2^{m_4}-2^{m_1}-2^{m_2})
\leq 
\frac{4}{3}c_{5} (1+\epsilon) L^4
\]
thus getting
\begin{equation}
\label{I-estim3}
I'
\leq
 (1+\epsilon) 
 c_{4} c_{5} \frac{\pi^2}{96}  
 \singseries^{\second}(q) 
\eta X L^4,
\end{equation}
for a sufficiently small $\epsilon$.
Hence, by \eqref{expanded1}-\eqref{diag-eval}
and \eqref{I-estim3}, we finally get
\[
I
<
(1+\epsilon)
\eta X L^4
\Bigl(
\log^{2} 2
+
c_{4}c_{5} \dfrac{\pi^2}{96}\singseries^{\second}(q) 
\Bigr),
\]
this way proving
Lemma \ref{sqroot-Dioph-equation}.
\end{proof}

We recall now a famous result by Ghosh about $S_{2}(\alpha)$.
\begin{lem}[Ghosh \cite{Ghosh1981}, Theorem 2]
\label{Ghosh-estim}
Let $\alpha$ be a real number and $a,q$ be positive integers
satisfying $(a,q)=1$ and $\vert \alpha -a/q \vert < q^{-2}$. 
Let moreover $\epsilon>0$.
Then  
\[
S_{2}(\alpha)
\ll_{\epsilon} 
X^{1/2+\epsilon}
\left(
\frac{1}{q}
+
\frac{1}{X^{1/4}}
+
\frac{q}{X}
\right)^{1/4}  .
\] 
\end{lem}

As an application of the previous lemma, we get the following result.
\begin{lem}
\label{minor-arc-lemma}Suppose that $\lambda_{2}/\lambda_{3}$ is irrational, let $X= q^{2}$ where $q$ is the denominator of a convergent of the continued fraction for $\lambda_{2}/\lambda_{3}$. Let 
$V(\alpha)= \min(\vert
S_{2}(\lambda_{2}\alpha)
\vert;
\vert
S_{2}(\lambda_{3}\alpha)
\vert).$
Then for arbitrarily small $\epsilon$ we have
\[
\sup_{\alpha \in \m}
V(\alpha)
\ll 
X^{7/16+\epsilon}.
\] 
\end{lem}
\begin{proof}
Let $\alpha \in \m$, $Q=X^{1/4}/(\log X)^{2}\leq P$. By Dirichlet Theorem, there exist integers $a_{i},q_{i}$  with $1\leq q_{i}\leq X/Q$, $(a_{i},q_{i})=1$, such that
$\vert \lambda_{i} \alpha q_{i}-a_{i}\vert \leq Q/X$, $i=2,3$.
We remark that $a_{2}a_{3} \neq 0$ otherwise we would have $\alpha\in \M$.
Now suppose that  $q_{i} \leq Q$, $i=2,3$. In this case we get
\[
a_{3}q_{2} \frac{\lambda_{2}}{\lambda_{3}} - a_{2}q_{3}
=
( \lambda_{2} \alpha q_{2}-a_{2}) \frac{a_{3}}{\lambda_{3} \alpha}
-
( \lambda_{3} \alpha q_{3}-a_{3}) \frac{a_{2}}{\lambda_{3} \alpha}
\]
and hence 
\[
\left\vert
a_{3}q_{2} \frac{\lambda_{2}}{\lambda_{3}} - a_{2}q_{3}
\right\vert
\leq 
2\left(
1+ \left\vert  \frac{\lambda_{2}}{\lambda_{3}} \right\vert 
\right) 
\frac{Q^{2}}{X}
<
\frac{1}{2q}
\]
for a sufficiently large $X$.
Then, from the law of best approximation and the definition of $\m$, we obtain
\[
X^{1/2}=q
\leq
\vert a_{3}q_{2} \vert
\ll q_{2}q_{3} \log ^{2} X
\leq Q^{2}\log ^{2} X
\leq X^{1/2}\log ^{-2} X.
\]

Hence either $q_{2}>Q$ or $q_{3}>Q$.
Assume, without loss of generality, that $q_{2}>Q$. 
Using Lemma \ref{Ghosh-estim} on $S_2(\lambda_2 \alpha)$, we have
\begin{align*}
V(\alpha)
\leq
\vert
S_2(\lambda_2 \alpha)
\vert
& \ll_{\epsilon} 
X^{1/2+\epsilon}
\sup_{Q<q_{2}\leq X/Q } 
\left(
\frac{1}{q_{2}}
+
\frac{1}{X^{1/4}}
+
\frac{q_{2}}{X}
\right)^{1/4}
\\
& 
\ll_{\epsilon} 
X^{7/16+\epsilon}
(\log X)^{1/2}
\end{align*}
thus proving Lemma \ref{minor-arc-lemma}.
\end{proof}

To estimate the contribution of $G(\alpha)$ on the minor arc
we use Pintz-Ruzsa's method as developed in
\cite{PintzR2003}, \S 3-7.
\begin{lem}[Pintz-Ruzsa \cite{PintzR2003}, \S7]
\label{minor-arc-power-of-two-estim}
Let $0< c <1$. Then there exists $\nu=\nu(c)\in (0,1)$ 
such that 
\[
\vert 
E(\nu) 
\vert
:=
\vert 
\{
\alpha \in (0,1) \
\textrm{such that} \ 
\vert 
G(\alpha) 
\vert 
> \nu L
\}
\vert
\ll_{M,\epsilon}
X^{-c}.
\] 
\end{lem}
To obtain explicit values for $\nu$ we used the version of Pintz-Ruzsa algorithm already 
implemented to get the results used in Languasco-Zaccagnini \cite{LanguascoZ2010b}.
We used the PARI/GP \cite{PARI2} scripting language and the
gp2c compiling tool to be able to compute fifty decimal digits
(but we write here just ten) of the constant
involved in the previous Lemma.
Running the program in our case, Lemma \ref{minor-arc-power-of-two-estim}
gives the following result:
\begin{equation}
\label{G-irrational}
\vert
G(\alpha)
\vert
\leq
0.8844472132  \cdot L
\end{equation}
if
$\alpha \in [0,1]  \setminus E$ where
$\vert E \vert \ll_{M,\epsilon} X^{-3/4-10^{-20}}$.
The computing time to get \eqref{G-irrational} 
on a Apple MacBook Pro was equal to
26 minutes and 28 seconds (but to get 30 correct digits just 
3 minutes and 31 seconds suffice).
You can download the PARI/GP source code of our program
together with the cited numerical values at the following link:
\url{www.math.unipd.it/~languasc/PintzRuzsaMethod.html}.

Now we state some lemmas we will use to work on the major arc.
Let $\theta(x)=\sum_{p\leq x} \log p$,
\begin{equation}
\label{Selberg-int-def}
J(X,h) 
=
\int_{\epsilon X}^X
(\theta(x+h)- \theta(x) -h)^2
\dx x
\end{equation}
and
\begin{equation}
\label{sqroot-Selberg-int-def}
J^{*}(X,h) 
=
\int_{\epsilon X}^X
\left(
\theta(\sqrt{x+h})- \theta(\sqrt{x}) -(\sqrt{x+h}-\sqrt{x})
\right)^2
\dx x
\end{equation}
be two different versions of the Selberg integral and
\[
  U_{1}(\alpha)
  =
    \sum_{\epsilon X\leq n \leq X} e(\alpha n)
\quad
\textrm{and}
\quad
  U_{2}(\alpha)
  =
    \sum_{\epsilon X\leq n^{2} \leq X} e(\alpha n^{2}).
\]

Applying a famous Gallagher's lemma 
on the truncated $L^2$-norm
of exponential sums to $S_1(\alpha) - U_1(\alpha) $, one gets the following
well-known statement which we cite from 
Br\"udern-Cook-Perelli \cite{BrudernCP1997}, Lemma 1.
\begin{lem}
\label{BCP-Gallagher}
For $1/X \leq Y \leq 1/2$ we have
\[
\int_{-Y}^Y
\vert
S_1(\alpha) - U_1(\alpha) 
\vert^2
\dx \alpha
\ll_{\epsilon}
\frac{\log^{2} X}{Y}
+
Y^2X
+
Y^2 J \Bigl( X,\frac{1}{2Y} \Bigr),
\]
where $J(X,h)$ is defined in \eqref{Selberg-int-def}.
\end{lem}
To estimate the Selberg integral, we use the next result.
\begin{lem}[Saffari-Vaughan \cite{SaffariV1977}, \S6]
\label{Saffari-Vaughan}
Let $\epsilon$ be an arbitrarily small positive constant.  
There exists a positive constant $c_{6}(\epsilon)$ such that
\[
J(X,h)
\ll_{\epsilon}
h^2X
\exp
\Big(
- c_{6}
\Big(
\frac{\log X}{\log \log X}
\Big)^{1/3}
\Big)
\]
uniformly for 
$X^{1/6+\epsilon} \leq h \leq X$. 
\end{lem}

In a similar way we can also prove the following
\begin{lem} 
\label{sqroot-BCP-Gallagher}
For $1/X \leq Y \leq 1/2$  we have 
\[
\int_{-Y}^Y
\vert
S_2(\alpha) - U_2(\alpha) 
\vert^2
\dx \alpha
\ll_{\epsilon}
\frac{\log^{2} X}{YX}
+
Y^2X
+
Y^2 J^{*} \Bigl( X,\frac{1}{2Y} \Bigr),
\]
where $J^{*}(X,h)$ is defined in \eqref{sqroot-Selberg-int-def}.
\end{lem}
\begin{proof}
Letting 
${\mathcal I} := \int_{-Y}^Y
\vert
S_2(\alpha) - U_2(\alpha) 
\vert^2
\dx \alpha$,   
we can write
\begin{align*}
{\mathcal I} 
&= 
\int_{-Y}^{Y}
\Big \vert
 \sum_{\epsilon X\leq p^{2} \leq X} 
 \log p\ e(p^{2}\alpha) 
 - 
 \sum_{\epsilon X\leq n^{2} \leq X} e(\alpha n^{2})
\Big\vert^{2}
\dx \alpha   
\\
&
= 
\int_{-Y}^{Y}
\Big \vert
 \sum_{\epsilon X\leq n^{2} \leq X} ( k(n)-1 )e(n^{2}\alpha)
\Big\vert^{2}
\dx \alpha, 
\end{align*}
where $k(n) = \log p$ if $n=p$ prime and $k(n)=0$ otherwise.  
By Gallagher's lemma (Lemma 1 of \cite{Gallagher1970})
we obtain
\[
{\mathcal I}  
\ll  
Y^{2}
\int_{-\infty}^{\infty}
\Bigl(
\sum_{\substack{x \leq n^{2} \leq x + H \\ \epsilon X\leq n^{2} \leq X}}
 ( k(n)-1)
\Bigr)^{2}
\dx x  
\]
where we defined 
$H=1/(2Y)$.
We can restrict the integration range
to 
$
E = \left[ \epsilon X - H, X\right]
$ since otherwise the inner sum is empty.
Moreover we split $E$ as $E=E_{1} \sqcup E_{2} \sqcup E_{3}$ where 
the symbol $\sqcup$ represents the disjoint union and
\(
E_{1}= \left [\epsilon X - H,  \epsilon X \right], 
\)
\(
E_{2}= \left [\epsilon X, X - H\right] ,
\)
\(
E_{3}= \left [X - H, X\right] .
\) 
Accordingly  we can write
\begin{equation}
\label{I123}
\begin{split}
{\mathcal I}  & \ll 
Y^{2}
\left(
\int_{E_{1}}
+
\int_{E_{2}}
+
\int_{E_{3}}
\right)
\Bigl(
\sum_{\substack{x \leq n^{2} \leq x + H \\ \epsilon X\leq n^{2} \leq X}}
 ( k(n)-1)
\Bigr)^{2}
\dx x
=
Y^{2}
(I_{1}+I_{2}+I_{3}),
\end{split}
\end{equation}
say.
We now proceed to estimate $I_{i}$, for every $i=1,2,3$.

\paragraph{\textbf{Estimation of $I_1$.}}
By trivial estimates we have
\begin{align}
\label{pretheta}
I_{1} 
& =
\int_{E_{1}}
\Bigl(
\sum_{ \epsilon X\leq n^{2} \leq x + H} ( k(n)-1)
\Bigr)^{2}
\dx x 
\nonumber \
\\
& =
\int_{\epsilon X - H}^{\epsilon X}
\Bigl(
\theta\left(\sqrt{x+H}\right) - \theta(\sqrt{\epsilon X}) - \left(\sqrt{x + H} - \sqrt{\epsilon X} \right) + \Odi{1}
\Bigr)^{2}
\dx x 
\nonumber \\
& \ll 
\int_{\epsilon X - H}^{\epsilon X}
\Bigl(
\theta\left(\sqrt{x+H}\right) - \theta(\sqrt{\epsilon X}) - \left(\sqrt{x + H} - \sqrt{\epsilon X} \right)
\Bigr)^{2}
\dx x 
+ H.
\end{align}
Using a trivial estimate in \eqref{pretheta}  we have
\begin{align}
\label{I1}
I_{1}  
\ll
\log^{2}X
\int_{\epsilon X - H}^{\epsilon X}
\Bigl(
\sqrt{x+H} - \sqrt{\epsilon X}
\Bigr)^{2}
\ \dx x 
+ H 
\ll_{\epsilon}
\frac{H^{3}\log^{2} X}{ X}
+ H,
\end{align}
where the last step follows applying the Mean Value Theorem to the integrand
function.

\paragraph{\textbf{Estimation of $I_3$.}} 
The estimation of $I_{3}$ is similar to the one of $I_{1}$.  We have 
\begin{align*}
I_{3} 
& =
\int_{E_{3}}
\Bigl(
\sum_{ x \leq n^{2} \leq X} ( k(n)-1)
\Bigr)^{2}
\dx x  \\
&
 \ll
\int_{X - H}^{X}
\Bigl(
\theta(\sqrt{X}) - \theta\left( \sqrt{x}\right) - \left( \sqrt{X} - \sqrt{x} \right)
\Bigr)^{2}
\dx x
+ H.
\end{align*}
Again using a trivial estimate and  the Mean Value Theorem
we get
\begin{align}
\label{I3}
I_{3} 
& \ll
\log^{2}X
\int_{X - H}^{X }
\Bigl(
\sqrt{X} - \sqrt{x}
\Bigr)^{2}
\dx x 
+ H  
\ll_{\epsilon}
\frac{H^{3}\log^{2}X}{ X}
+ H.
\end{align}

\paragraph{\textbf{Estimation of $I_2$.}} 
We have
\begin{align}
\label{I2}
I_{2}
& =
\int_{E_{2}}
\Bigl(
\sum_{ x \leq n^{2} \leq x + H } ( k(n)-1)
\Bigr)^{2}
\dx x 
\nonumber 
\\
&
 \ll
\int_{\epsilon X }^{X}
\Bigl(
\theta \left( \sqrt{x + H} \right) - \theta \left( \sqrt{x} \right) 
-
\left( \sqrt{x + H}  - \sqrt{ x}\right)
\Bigr)^{2}
\dx x
+ 
{X} 
\nonumber \\
& =
J^{*}\left(X,H\right) +X,
\end{align}
where we used the definition in \eqref{sqroot-Selberg-int-def}. 
Therefore, by \eqref{I123},  
\eqref{I1}-\eqref{I2}, $Y\geq 1/X$ 
and recalling
$H=1/(2Y)$, we have
\begin{align*}
{\mathcal I}
&\ll_{\epsilon} 
 \frac{\log^{2}X}{XY} +  XY^{2} + Y^{2} J^{*}\left(X,\frac{1}{2Y}\right)
\end{align*} 
and this proves Lemma \ref{sqroot-BCP-Gallagher}. 
\end{proof}
To estimate $J^{*}(X,h)$, we use the next result.
\begin{lem}
\label{sqroot-Saffari-Vaughan}
Let $\epsilon$ be an arbitrarily small positive constant.  
There exists a positive constant $c_{6}(\epsilon)$ such that
\[
J^{*}(X,h)
\ll_{\epsilon}
h^2
\exp
\Big(
- c_{6} 
\Big(
\frac{\log X}{\log \log X}
\Big)^{1/3}
\Big)
\]
uniformly for  $X^{7/12+\epsilon} \leq h \leq X$.  
\end{lem}
\begin{proof}
We reduce our problem to estimate
\begin{equation}
\label{def_Jpsi}
J^{*}_{\psi}(X,h)
:=
\int_{\epsilon X}^X
\left(
\psi(\sqrt{x+h})-\psi(\sqrt{x}) -
(\sqrt{x+h}-\sqrt{x})
\right)^2
\dx x 
\end{equation}
since,
using $|a+b|^{2} \leq 2|a|^{2}+2|b|^{2}$, it is easy to see that
\begin{align*} 
J^{*}(X,h) 
&\ll
J^{*}_{\psi}(X,h)
\\
&+
\int_{\epsilon X}^X
\left(
\psi(\sqrt{x+h})- \psi(\sqrt{x}) -
\theta(\sqrt{x+h})+\theta(\sqrt{x}) 
\right)^{2}
\dx x.
\end{align*}
By a trivial estimate and the Mean Value Theorem we obtain
\begin{align}
\label{fromThetatoPsiFinale}
J^{*}(X,h) 
&\ll_{\epsilon}
J^{*}_{\psi}(X,h) 
+
\int_{\epsilon X}^{X} 
 \frac{h^{2}}{X^{3/2}} \log^{4} X \
\dx x 
\ll_{\epsilon}
J^{*}_{\psi}(X,h) 
+
h^{2} \frac{\log^{4} X}{X^{1/2}}.
\end{align}

To estimate the right hand side of \eqref{fromThetatoPsiFinale}, we use the following result we will prove later.

\begin{lem} 
\label{Lemma_da_h_a_deltax}
Let $\epsilon$ be an arbitrarily small positive constant.  
There exists a positive constant $c_{6}(\epsilon)$ such that
\[
J^{*}_{\psi}(X,h)
\ll_{\epsilon}
h^2
\exp
\Big(
- c_{6}
\Big(
\frac{\log X}{\log \log X}
\Big)^{1/3}
\Big)
\]
uniformly for  $X^{7/12+\epsilon} \leq h \leq X$, where $J^{*}_{\psi}(X,h)$ is defined in \eqref{def_Jpsi}.
\end{lem}

Therefore, by \eqref{fromThetatoPsiFinale} and Lemma \ref{Lemma_da_h_a_deltax}, we obtain
\[
J^{*}(X,h)
 \ll_{\epsilon} 
h^{2}
\exp
\Big(
- c_{6}
\Big(
\frac{\log X}{\log \log X}
\Big)^{1/3}
\Big)
\]
thus proving Lemma \ref{sqroot-Saffari-Vaughan}.
\end{proof}
Lemma  \ref{Lemma_da_h_a_deltax} will follow from the
following lemma.
\begin{lem} 
\label{Lemma_deltax}
Let $\epsilon$ be an arbitrarily small positive constant.  
There exists a positive constant $c_{6}(\epsilon)$ such that 
\begin{align*}
\widetilde{J}^{*}_{\psi}(X,\delta)
&:=
\int_{\epsilon X}^X
\left(
\psi(\sqrt{x+\delta x})-\psi(\sqrt{x}) -
(\sqrt{x+\delta x}-\sqrt{x})
\right)^2
\dx x 
\\
&\ll_{\epsilon}
\delta^{2}X^2
\exp
\Big(
- c_{6}
\Big(
\frac{\log X}{\log \log X}
\Big)^{1/3}
\Big)
\end{align*}
uniformly for $ X^{-5/12+\epsilon} \leq \delta \leq 1$. 
\end{lem}

\begin{proof}[Proof of Lemma \ref{Lemma_deltax}]
We follow the argument of  \S 5 in Saffari-Vaughan \cite{SaffariV1977}.
To estimate $\widetilde{J}^{*}_{\psi}(X,\delta)$, we use the truncated explicit formula for $\psi(x)$ (see, \emph{e.g.}, eq.~(9)-(10) of \S17 of Davenport \cite{Davenport2000}):
\begin{equation*}
\psi(x)
=
x - \sum_{\vert \gamma \vert \leq T} \frac{x^{\rho}}{\rho} + 
\Odi{\frac{x}{T} \log^{2}(xT)
+\log x}
\end{equation*}
uniformly in $T\geq 2$ and for $\rho = \beta + i \gamma$ 
non-trivial zeros of $\zeta(s)$. So
\begin{align} 
\label{stimainizialeJpsi}
\widetilde{J}^{*}_{\psi}(X,\delta)
& \ll 
\int_{\epsilon X}^X
\Bigl \vert
 \sum_{\substack{\vert \gamma \vert \leq T\\ \beta \geq 1/2}} 
 x^{\rho/2}  \frac{((1+\delta)^{\rho/2} -1)}{\rho}
 \Bigr\vert^2
\dx x 
+\frac{X^{2}}{T^{2}} \log^{4}(XT)
+X \log^{2}X.
\end{align}

As in page 316 of Ivi\'{c} \cite{Ivic1985}, we define
\( 
c(\delta, \rho) =  ((1+\delta)^{\rho}-1)/\rho,
\) 
and  remark
\begin{equation}
\label{stimac}
\left \vert c\left(\delta, \frac{\rho}{2}\right)\right \vert  
\ll 
\min \left( \frac{1}{ \vert \gamma \vert }; \delta \right).
\end{equation} 
Assuming $T \geq 1/\delta$, we can split the summation in 
\eqref{stimainizialeJpsi} in two cases defined accordingly to 
\eqref{stimac}.
We obtain 
\begin{align}
\label{JpsiSommaDiA}
\widetilde{J}^{*}_{\psi}(X,\delta)
& \ll
A_{[0,  1/\delta)} + A_{[1/\delta,T]} + \frac{X^{2}}{T^{2}} \log^{4}(XT)
+ X(\log X)^{2},
\end{align}
with 
\begin{align}
\label{stimaAI}
A_{I} 
&= 
\int_{\epsilon X}^X
\Bigl \vert
 \sum_{\substack{\vert \gamma \vert \in I \\ \beta \geq 1/2}} x^{\rho/2} c\left(\delta, \frac{\rho}{2}\right)
\Bigr \vert^2
\dx x
\nonumber   
\\  & 
=
 \sum_{\substack{\vert \gamma_{1} \vert \in I \\ \beta_{1} \geq 1/2}}
 \sum_{\substack{\vert \gamma_{2} \vert \in I \\ \beta_{2} \geq 1/2}}
c\Bigl(\delta, \frac{\rho_{1}}{2}\Bigr) c\Bigl(\delta, \frac{\overline{\rho}_{2}}{2}\Bigr)
\frac{2X^{(\rho_{1} +\overline{\rho}_{2})/2 +1} (1- \epsilon^{(\rho_{1} +\overline{\rho}_{2})/2 +1})}{\rho_{1} +\overline{\rho}_{2}+2}
\nonumber \\  
 &\ll
 \sum_{\substack{\vert \gamma_{1} \vert \in I \\ \beta_{1} \geq 1/2}}
 \sum_{\substack{\vert \gamma_{2} \vert \in I \\  1/2 \leq \beta_{2} \leq \beta_{1}}}
\Bigl\vert c\Bigl(\delta, \frac{\rho_{1}}{2}\Bigr)\Bigr\vert 
\Bigl\vert c\Bigl(\delta, \frac{\overline{\rho}_{2}}{2}\Bigr)\Bigr\vert 
\frac{X^{\beta_{1} +1}}{1+ \vert \gamma_{1} -\gamma_{2} \vert }.
\end{align} 
Now we deal separately  with $A_{[0, 1/\delta)}$ and $A_{[1/\delta, T]}$.

\paragraph{\textbf{Estimation of $A_{[0, 1/\delta)}$.}} 
From \eqref{stimac} and \eqref{stimaAI} we can write
\begin{align}
\notag
A_{[0,1/\delta)}
& \ll
\delta^{2}X
\sum_{\substack{\vert \gamma_{1} \vert < \delta^{-1}\\ \beta_{1} \geq 1/2}}
X^{\beta_{1}}
 \sum_{\substack{\vert \gamma_{2} \vert < \delta^{-1} \\  1/2 \leq \beta_{2} \leq \beta_{1}}}
\frac{1}{1+ \vert \gamma_{1} -\gamma_{2} \vert }
\\
\label{Amin1suDelta}
&
\ll
\delta^{2}X(\log X)^{2}
\sum_{\substack{\vert \gamma_{1} \vert < \delta^{-1}\\ \beta_{1} \geq 1/2}}
X^{\beta_{1}},
\end{align}
where the last inequality follows from
\begin{align}
\label{sommesu1sugamma}
\sum_{\substack{\vert \gamma_{2} \vert < \delta^{-1}\\  1/2 \leq \beta_{2} \leq \beta_{1}}}
\frac{1}{1+ \vert \gamma_{1} -\gamma_{2} \vert }
& \ll 
\sum_{n=0}^{2/\delta} 
\frac{\log (\gamma_{1}+n)}{1+n} 
\ll 
\log \left(\frac{3}{\delta}\right)^{2}  
\ll
(\log X)^{2}
\end{align}
in which we used  
the Riemann-Von Mangoldt formula  
and $\delta > X^{-1}$. 
Denoting by $S_{[0, 1/\delta)}$ the sum in the right hand side of \eqref{Amin1suDelta}, we get 
\begin{align} 
S_{[0, 1/\delta)}
&:=
\sum_{\substack{\vert \gamma \vert <  1/\delta\\ \beta \geq 1/2}}
X^{\beta}  
\ll
\log X
\max_{1/2 \leq u \leq 1} 
\Bigl(
X^{u} N\Bigl(u,\frac{1}{\delta}\Bigr) 
\Bigr). \nonumber
\end{align} 
We recall the Ingham-Huxley zero-density estimate: 
for  $\frac{1}{2} \leq \sigma \leq 1$ we have
 that $N(\sigma,t) \ll t^{(12/5)(1-\sigma)} (\log t)^{B}$
 and   the Vinogradov-Korobov zero-free region: 
there are no zeros $\beta+i\gamma$ of the Riemann zeta function having 
\[
\beta 
\geq 
1- \frac{c_{7}}{(\log ( \vert \gamma \vert +2))^{2/3}(\log \log ( \vert \gamma \vert +2))^{1/3}},
\] 
where $c_{7}>0$ is an absolute constant. 
In the following $c_{7}$ will not necessarily be the same at each occurrence. 
Here we have $ \vert \gamma \vert \leq T$, and so 
$
N(u, t)=0 $
for every  $t\leq T$ and $u\geq 1-K$ with 
\[
K =\frac{c_{7}}{(\log T)^{2/3}(\log \log T)^{1/3}}.
\]
From the previous remarks, we obtain
\begin{align*}
S_{[0, 1/\delta)}
& \ll
\log X
\max_{1/2 \leq u \leq 1 - K} 
\left(
(\delta^{-1})^{(12/5)(1-u)} (\log({\delta}^{-1}))^{B}X^{u}
\right)
\\
& \ll
(\log X)^{B+1} \delta^{-\frac{12}{5}}
\max_{1/2 \leq u \leq 1 - K} 
\left(
(\delta^{12/5}X)^{u}
\right),
\end{align*}
since $\delta > X^{-1}$. 
The maximum is attained at $u=1-K$ and so
\begin{align} 
S_{[0, 1/\delta)}
& \ll
(\log X)^{B+1} \delta^{-\frac{12}{5}}
\delta^{\frac{12}{5}(1-K)}X^{1-K}
\nonumber  
=
X (\log X)^{B+1} (\delta^{12/5}X)^{-K}.
\nonumber  
\end{align}
Inserting the last estimate into \eqref{Amin1suDelta}, we can write
\begin{equation}
\label{FinAmin1suDelta}
A_{[0, 1/\delta)}
\ll
\delta^{2}X^{2}
(\log X)^{B+3}
(\delta^{12/5}X)^{-K}.
\end{equation}

\paragraph{\textbf{Estimation of $A_{[1/\delta,T]}$.}}  
From \eqref{stimac} and  \eqref{stimaAI} we get
\begin{align} 
A_{[1/\delta,T]}
& \ll
X
\sum_{\substack{1/\delta \leq \vert \gamma_{1} \vert \leq T \\ \beta_{1} \geq 1/2}}
\frac{X^{\beta_{1}}}{ \vert \gamma_{1} \vert }
\sum_{\substack{1/\delta \leq \vert \gamma_{2} \vert \leq T \\  1/2 \leq \beta_{2} \leq \beta_{1}}}
\frac{1}{ \vert \gamma_{2} \vert (1+ \vert \gamma_{1} -\gamma_{2} \vert )}
\nonumber \\
&\ll
X
\sum_{\substack{1/\delta \leq \vert \gamma_{1} \vert \leq T \\ \beta_{1} \geq 1/2}}
\frac{X^{\beta_{1}}}{ \vert \gamma_{1} \vert ^{2}}
\sum_{\substack{\vert \gamma_{1} \vert  \leq \vert \gamma_{2} \vert \leq T\\  1/2 \leq \beta_{2} \leq \beta_{1}}}
\frac{1}{1+ \vert \gamma_{1} -\gamma_{2} \vert }
\nonumber  
\\
&\ll
X(\log T)^{2}
\sum_{\substack{1/\delta \leq \vert \gamma_{1} \vert \leq T  \\ \beta_{1} \geq 1/2}}
\frac{X^{\beta_{1}}}{ \vert \gamma_{1} \vert ^{2}}, \nonumber
\end{align}
where the last step follows from \eqref{sommesu1sugamma} with $T$ instead of $1/\delta$. 
By a simple trick, we can rewrite the previous inequality as
\begin{equation}
\label{A-spezzata}
A_{[1/\delta,T]}
\ll
X(\log T)^{2}
(
S'_{[1/\delta,T]} + S''_{[1/\delta,T]}
)
\end{equation}
with 
$$
S'_{[1/\delta,T]} 
=
\sum_{\substack{1/\delta \leq \vert \gamma \vert \leq T  \\ \beta \geq 1/2}}
X^{\beta}
\Big(
\frac{1}{ \vert \gamma \vert ^{2}} -\frac{1}{T^{2}}
\Big)
\quad 
\textrm{and}
\quad
S''_{[1/\delta,T]} 
=
\frac{1}{T^{2}}
\sum_{\substack{1/\delta \leq \vert \gamma \vert \leq T  \\ \beta \geq 1/2}}
X^{\beta}.
$$

For $S''_{[1/\delta,T]}$ we can argue as we did for $S_{[0, 1/\delta)}$, just keeping in mind that this time $1/\delta \leq  \vert \gamma \vert  \leq T$. Hence
\begin{align*}
S''_{[1/\delta,T]} 
& \ll
\frac{\log X}{T^{2}} 
\max_{1/2 \leq u \leq 1- K} 
\Bigl(
X^{u}
\Bigl[
N(u,T) - N\Bigl(u,\frac{1}{\delta}\Bigr)
\Bigr]
\Bigr) .
\end{align*} 

Concerning $S'_{[1/\delta,T]} $ we immediately obtain
\begin{align*}
S'_{[1/\delta,T]} 
& =
\sum_{\substack{1/\delta \leq \vert \gamma \vert \leq T  \\ \beta \geq 1/2}}
X^{\beta}
\int_{ \vert \gamma \vert }^{T} 
\frac{2}{t^{3}}\dx t
 =
2
\int_{1/\delta}^{T} 
\bigg(
\sum_{\substack{1/\delta \leq \vert \gamma \vert \leq t  \\ \beta \geq 1/2}}
X^{\beta}
\bigg)
\frac{\dx t}{t^{3}}.
\end{align*}
Using $t \leq T$,  we can write
\begin{align*}
S'_{[1/\delta,T]} 
& \ll
\log X
\int_{1/\delta}^{T} 
\max_{1/2 \leq u \leq 1- K} 
\Bigl(
X^{u}
\Bigl[
N(u,t) - N\Bigl(u,\frac{1}{\delta}\Bigr)
\Bigr]
\Bigr)
\frac{\dx t}{t^{3}}.
\end{align*}
Therefore
\begin{align*}
S'_{[1/\delta,T]} + S''_{[1/\delta,T]} 
& \ll
\log X \log (T \delta)
\\
&\times
\max_{1/\delta \leq t \leq T}
\bigg(
\frac{1}{t^{2}}
\max_{1/2 \leq u \leq 1- K} 
\Bigl(
X^{u}\,
t^{(12/5)(1-u)}(\log t)^{B}
\Bigr)
\bigg),
\end{align*}
by the Ingham-Huxley zero-density estimate. So, by \eqref{A-spezzata}, this estimation and $t \leq T$, we get
\begin{align} 
A_{[1/\delta,T]}
& \ll
X(\log T)^{B+2}\log X \log (T \delta)
\max_{1/2 \leq u \leq 1- K} 
\Bigl(
X^{u}
\max_{1/\delta \leq t \leq T}
\left(
t^{(12/5)(1-u)-2}
\right)
\Bigr). \nonumber
\end{align}
To compute the inner maximum above, we just remark that 
$(12/5)(1-u)-2 < 0 $ (which holds for $u > 1/6$),
and hence it is attained at $t=1/\delta$.
So 
\begin{align} 
A_{[1/\delta,T]}
& \ll
X(\log T)^{B+2}\log X \log (T \delta)
\max_{1/2 \leq u \leq 1- K} 
\left(
X^{u}
(\delta^{-1})^{(12/5)(1-u)-2}
\right)
\nonumber \\ 
& =
\delta^{-\frac{2}{5}} X 
(\log T)^{B+2}\log X \log (T \delta)
\max_{1/2 \leq u \leq 1- K} 
\left(
( X \delta^{12/5})^{u}
\right). \nonumber
\end{align} 
The maximum is attained at $u=1-K$, thus
\begin{align}
\notag
A_{[1/\delta,T]}
& \ll
\delta^{-\frac{2}{5}} X 
(\log T)^{B+2}\log X \log (T \delta)
(X \delta^{12/5})^{1-K}
\\
\label{A-int-alto}
&=
\delta^{2}X^{2} (\log T)^{B+2}\log X \log (T \delta)
(X \delta^{12/5})^{-K}. 
\end{align} 

\paragraph{\textbf{Conclusion of the proof.}}
Inserting \eqref{FinAmin1suDelta} and \eqref{A-int-alto}  into 
\eqref{JpsiSommaDiA}, we get
\begin{align}
\notag
\widetilde{J}^{*}_{\psi}(X,\delta)
& \ll
\delta^{2}X^{2}
(X\delta^{12/5})^{-K}
\log X  
\Big(
(\log X)^{B+2}
+
(\log T)^{B+2} \log (T \delta)
\Big)
\\ 
\label{almost-done-J-star}
& \quad \quad +
\frac{X^{2}}{T^{2}} (\log(XT))^{4}
+ X (\log X)^{2}.
\end{align}
Choosing $T\leq X^{1/2}$ we have  
\[
K =
\frac{c_{7}}{(\log T)^{2/3}(\log \log T)^{1/3}} 
\geq
\frac{c_{8}}{(\log X)^{2/3}(\log \log X)^{1/3}} , 
\] 
for a suitable positive constant $c_{8}$. Taking
now $T \geq X^{5/12-\epsilon}(X\delta^{12/5})^{K/2}$ $\times (\log X)^{-B/2}$
and recalling $\delta>X^{-5/12+\epsilon}$,
equation
\eqref{almost-done-J-star} becomes
\begin{equation}
\notag
\widetilde{J}^{*}_{\psi}(X,\delta)
\ll
\delta^{2}X^{2}
(X\delta^{12/5})^{-K}
(\log X)^{B+4}
\end{equation}
since the conditions on $T$ are compatible. 
Hence we immediately obtain
\begin{align*}
\widetilde{J}^{*}_{\psi}(X,\delta)
& \ll  
\delta^{2}X^{2} (\log X)^{B+4}
\exp
\Big(
- \frac{c_{8}(\log X+ (12/5)\log\delta)}{(\log X)^{2/3}(\log \log X)^{1/3}}
\Big)
\\
& \ll \delta^{2}X^{2}  
\exp
\Big(
- c_{9}
\Big(
\frac{\log X}{\log \log X}
\Big)^{1/3}
\Big)
\end{align*} 
for a sufficiently large $X$ and $c_{9}=c_{9}(\epsilon)$. 
Hence Lemma \ref{Lemma_deltax} is proved.
\end{proof}
\begin{proof}[Proof of Lemma \ref{Lemma_da_h_a_deltax}]
We follow the argument of \S 6 in Saffari-Vaughan \cite{SaffariV1977}. 
Let now  $2h \leq v \leq 3h$.
To estimate $J^{*}_{\psi}(X,h)$ (defined in \eqref{def_Jpsi}), we first remark
\begin{align}
\notag
h
J^{*}_{\psi} & (X,h) 
 \ll
\int_{\epsilon X}^{X} 
\int_{2h}^{3h} 
\big(
\psi(\sqrt{x+v})-\psi(\sqrt{x}) -
(\sqrt{x+v}-\sqrt{x})
\big)^2
\dx v \
\dx x 
\\
&
 +
 \int_{\epsilon X}^{X} 
\int_{2h}^{3h} 
\big(
\psi(\sqrt{x+v})-\psi(\sqrt{x+h}) -
(\sqrt{x+v}-\sqrt{x+h})
\big)^2   
\dx v \
\dx x .
\label{int_da_studiare}
\end{align}
Setting $z=v-h, y = x+h$ and changing variables in the last integration,
the right hand side of  \eqref{int_da_studiare} becomes
\begin{align*}
 & \ll
\int_{\epsilon X}^{X} 
\int_{2h}^{3h}
\big(
\psi(\sqrt{x+v})-\psi(\sqrt{x}) -
(\sqrt{x+v}-\sqrt{x})
\big)^2
\dx v \
\dx x 
\\
& \quad 
 +
 \int_{\epsilon X+h}^{X+h} 
 \int_{h}^{2h}
\big(
\psi(\sqrt{y+z})-\psi(\sqrt{y}) -
(\sqrt{y+z}-\sqrt{y})
\big)^2 
\dx z \
\dx y.
\end{align*}
Since both the integrand functions are non-negative, we can extend the integration ranges merging $x$ with $y$ and $v$ with $z$. Hence
\begin{align*}
h
J^{*}_{\psi}& (X,h)  
\ll
\int_{\epsilon X}^{X+h} 
\int_{h}^{3h}
\big(
\psi(\sqrt{x+v})-\psi(\sqrt{x}) -
(\sqrt{x+v}-\sqrt{x})
\big)^2
\dx v \
\dx x 
\\
& =
\int_{\epsilon X}^{X+h} 
x
\int_{h/x}^{3h/x}
\big(
\psi(\sqrt{x+x\delta})-\psi(\sqrt{x}) -
(\sqrt{x+x\delta}-\sqrt{x})
\big)^2
\dx \delta \ 
\dx x,
\end{align*}
where in the last step we made the change of variable $\delta = v/x$,
thus getting $\delta \geq h/x \geq X^{-5/12 + \epsilon} $ as in the hypothesis of Lemma \ref{Lemma_deltax}.
Interchanging the integration order we obtain
\begin{align*}
h
J^{*}_{\psi}& (X,h)  
\ll
(X+h)\\
&\times
\int_{h/(X+h)}^{3h/(\epsilon X)} 
\int_{\epsilon X}^{X+h}
\big(
\psi(\sqrt{x+x\delta})-\psi(\sqrt{x}) -
(\sqrt{x+x\delta}-\sqrt{x})
\big)^2
\dx x  \
\dx \delta.
\end{align*}
Finally, using Lemma \ref{Lemma_deltax}, we get
\begin{align*}
J^{*}_{\psi}(X,h)
& \ll_{\epsilon}
\frac{X+h}{h}
\int_{h/(X+h)}^{3h/(\epsilon X)} 
\delta^{2}X^2
\exp
\Big(
- c_6
\Big(
\frac{\log X}{\log \log X}
\Big)^{1/3}
\Big) 
\dx \delta
\\
&
\ll_{\epsilon}
h^{2}
\exp
\Big(
- c_6
\Big(
\frac{\log X}{\log \log X}
\Big)^{1/3}
\Big).
\end{align*}
This concludes the proof of Lemma \ref{Lemma_da_h_a_deltax}.
\end{proof}
\section{The major arc}
\label{major-arc-eval}
Letting
\begin{equation}
\label{T1-def-estim}
  T_1(\alpha) 
  =
 \int_{\epsilon X}^{X}e(t\alpha)\dx t
  \ll_{\epsilon}
  \min \Bigl(X; \frac{1}{\vert\alpha\vert} \Bigr),
\end{equation}
and
\begin{equation}
\label{T2-def-estim}
  T_2(\alpha) 
  =
 \int_{(\epsilon X)^{1/2}}^{X^{1/2}}e(t^{2}\alpha)\dx t
 =
 \frac{1}{2}
  \int_{\epsilon X}^{X}v^{-1/2}e(v\alpha)\dx v
  \ll_{\epsilon}
X^{-1/2}  \min \Bigl(X; \frac{1}{\vert\alpha\vert} \Bigr),
\end{equation}
we first  write
\begin{align}
\notag
&I(X ; \M)
=
\int_\M
T_1(\lambda_1 \alpha)
T_2(\lambda_2 \alpha)
T_2(\lambda_3 \alpha)
\prod_{i=1}^{s} G(\mu_i\alpha) 
e(\varpi\, \alpha)
K(\alpha, \eta) \dx \alpha
\\
\notag
& +
\int_\M
(S_1(\lambda_1 \alpha) - T_1(\lambda_1 \alpha))
T_2(\lambda_2 \alpha)
T_2(\lambda_3 \alpha)
\prod_{i=1}^{s} G(\mu_i\alpha)
e(\varpi\, \alpha)
K(\alpha, \eta) \dx \alpha 
\\
\notag
& +
\int_\M
S_1(\lambda_1 \alpha)
(S_2(\lambda_2 \alpha) - T_2(\lambda_2 \alpha))
T_2(\lambda_3 \alpha)
\prod_{i=1}^{s} G(\mu_i\alpha)
e(\varpi\, \alpha)
K(\alpha, \eta) \dx \alpha 
\\
\notag
& +
\int_\M
S_1(\lambda_1 \alpha)
S_2(\lambda_2 \alpha)
(S_2(\lambda_3 \alpha) - T_2(\lambda_3 \alpha))
\prod_{i=1}^{s} G(\mu_i\alpha)
e(\varpi\, \alpha)
K(\alpha, \eta) \dx \alpha 
\\
\label{I-splitting}
&\hskip 1cm =
J_1+ J_2+J_3+J_4,
\end{align}
say.
In what follows we will prove that
\begin{equation}
\label{J1-lower-bound}
J_1
\geq 
\frac{(3-2\sqrt{2})\eta^2 X L^s}
{4
\left(
\vert \lambda_1 \vert + \vert \lambda_2\vert +\vert \lambda_3 \vert
\right)
}
+\Odip{\epsilon}{\eta^2  X^{1/5} L^{s+2}}
\end{equation} 
and
\begin{equation}
\label{J2-estim}
J_2 + J_3 +J_4
=
\odi{\eta^2XL^s},
\end{equation}
thus obtaining by \eqref{I-splitting}-\eqref{J2-estim} that
\[
I(X ; \M)
\geq 
\frac{3-2\sqrt{2}-\epsilon}
{4
\left(
\vert \lambda_1 \vert + \vert \lambda_2\vert +\vert \lambda_3 \vert
\right)
}
\eta^2 X L^s,
\]
proving that \eqref{major-goal} holds with
$c_1=
2^{-2}(3-2\sqrt{2}-\epsilon) 
\left( 
\vert \lambda_1 \vert + \vert \lambda_2\vert +\vert \lambda_3 \vert \right)^{-1}$ 
and $\epsilon>0$ is an arbitrarily small constant.

We will need the following estimates. The first one is a consequence of the 
Prime Number Theorem:
\begin{equation} 
\label{S1-squared}
\int_{0}^{1}
\vert
S_{1}(\alpha)
\vert^2
\dx \alpha
\ll_{\epsilon}
X \log X,
\end{equation}
while the second one is based on  Satz~3, page 94, of Rieger \cite{Rieger1968},
see also  the estimate of $H_{12}$ at page 106 of T.~Liu \cite{Liu2004}: 
\begin{equation} 
\label{S2-fourth}
\int_0^{1}
\vert
S_2(\alpha)
\vert^4
\dx \alpha
\ll_{\epsilon} X\log^{2} X.
\end{equation}

\paragraph{\textbf{Estimation of $J_2$, $J_3$ and $J_4$.}}

\mbox{}
\par
We first estimate $J_4$.
We remark that, by Euler's summation formula, we have
\begin{equation}
\label{T-U}
T_i(\alpha)  - U_i(\alpha) \ll 1 + X\vert \alpha \vert
\quad
\textrm{for every}
\quad i=1,2.
\end{equation}
So, by  
\eqref{dissect-def}, the Cauchy-Schwarz inequality,
and \eqref{S1-squared}-\eqref{T-U}
we get
\begin{align*}
\int_{\M}
\vert S_1&(\lambda_1\alpha) \vert
\vert S_2(\lambda_2\alpha) \vert
\vert
T_2(\lambda_3\alpha)
-
U_2(\lambda_3\alpha)
\vert
\dx \alpha
\\
& \ll_{\bm{\lambda}}
\int_{-1/X}^{1/X}
\vert S_1(\lambda_1\alpha) \vert
\vert S_2(\lambda_2\alpha) \vert
\dx \alpha
+
X 
\int_{1/X}^{P/X}
\vert \alpha \vert
\vert S_1(\lambda_1\alpha) \vert
\vert S_2(\lambda_2\alpha) \vert
\dx \alpha
\\
&
\ll_{\bm{\lambda}}
X^{-1/4}
\Bigl(
\int_{0}^{1}
\vert
S_{1}(\alpha)
\vert^2
\dx \alpha
\Bigr)^{1/2}
\Bigl(
\int_{0}^{1}
\vert
S_{2}(\alpha)
\vert^4
\dx \alpha
\Bigr)^{1/4}
\\
&
\hskip1cm
+
X
\Bigl(
\int_{1/X}^{P/X}
 \alpha^4
\dx \alpha
\Bigr)^{1/4}
\Bigl(
\int_{0}^{1}
\vert
S_{2}(\alpha)
\vert^4
\dx \alpha
\Bigr)^{1/4}
\Bigl(
\int_{0}^{1}
\vert
S_{1}(\alpha)
\vert^2
\dx \alpha
\Bigr)^{1/2}
\\
&
\ll_{\bm{\lambda},\epsilon}
X^{1/2} \log X
+
P^{5/4}
X^{1/2} 
\log X
=
\odi{X}
\end{align*}
since  $P=X^{2/5}/\log X$. 
Hence,
using the  trivial estimates
$\vert G(\mu_i \alpha)\vert \leq L$, $K(\alpha,\eta)\ll \eta^2$, we can write
\begin{align*}
J_4
&=
\int_\M
S_1(\lambda_1 \alpha)
S_2(\lambda_2 \alpha)
\Bigl(S_2(\lambda_3 \alpha) - U_2(\lambda_3 \alpha)\Bigr)
\prod_{i=1}^{s} G(\mu_i\alpha)
e(\varpi\, \alpha)
K(\alpha, \eta) \dx \alpha 
\\
&
+
\odip{\bm{\lambda},M,\epsilon}
{\eta^2X L^{s}}.
\end{align*}

Now using \eqref{dissect-def},  
$\vert S_2(\lambda_2\alpha) \vert \ll X^{1/2}$,
the Cauchy-Schwarz inequality, \eqref{S1-squared},
Lemmas \ref{sqroot-BCP-Gallagher}-\ref{sqroot-Saffari-Vaughan}
with  $Y=P/X$, and again the  trivial estimates
$\vert G(\mu_i \alpha)\vert \leq L$, $K(\alpha,\eta)\ll \eta^2$,
we have that
\begin{align*}
J_4
&
\ll
\eta^2 L^s X^{1/2}
\Bigl(
\int_{\M}
\vert
S_{2}(\lambda_3\alpha)
-
U_2(\lambda_3\alpha)
\vert^2
\dx \alpha
\Bigr)^{1/2}
\Bigl(
\int_{\M}
\vert
S_{1}(\lambda_1\alpha)
\vert^2
\dx \alpha
\Bigr)^{1/2} 
\\
&\hskip1cm +
\odip{\bm{\lambda}, M, \epsilon}
{\eta^2XL^{s}}
\\
&
\ll_{\bm{\lambda}, M, \epsilon}
\eta^2 L^s
X^{1/2}
\Bigl(
\int_{0}^{1}
\vert
S_{1}(\alpha)
\vert^2
\dx \alpha
\Bigr)^{1/2}
\exp
\Big(
- \frac{c_{6}(\epsilon)}{2}
\Big(
\frac{\log X}{\log \log X}
\Big)^{1/3}
\Big)
\\
&\hskip1cm +
\odip{\bm{\lambda}, M}
{\eta^2XL^{s}}
\\
&
\ll_{\bm{\lambda},M,\epsilon}
\eta^2 X L^{s+1/2}
\exp
\Big(
- \frac{c_{6}(\epsilon)}{2}
\Big(
\frac{\log X}{\log \log X}
\Big)^{1/3}
\Big)
=
\odi{\eta^2XL^s}.
\end{align*}

The integral $J_3$ can be estimated analogously using \eqref{T2-def-estim}
instead of $\vert S_2(\lambda_3\alpha) \vert \ll X^{1/2}$.

For $J_2$ we argue as follows. First of all, using again \eqref{T-U} 
and \eqref{T2-def-estim}
for every $i=2,3$,
we get
\begin{align*}
\int_{\M}
\vert
T_1(\lambda_1\alpha)
& -
U_1(\lambda_1\alpha)
\vert
\vert T_2(\lambda_2\alpha) \vert
\vert T_2(\lambda_3\alpha) \vert
\dx \alpha 
\\
& 
\ll_{\bm{\lambda}}
X
\int_{-1/X}^{1/X}
\dx \alpha
+
\int_{1/X}^{P/X}
\frac{X \vert  \alpha \vert}{X \alpha^2} 
\dx \alpha
\ll_{\bm{\lambda}}
1
+
\log P 
=
\odi{X}
\end{align*}
since $P=X^{2/5}/\log X$.
Hence,
using the  trivial estimates
$\vert G(\mu_i \alpha)\vert \leq L$, $K(\alpha,\eta)\ll \eta^2$, we can write
\begin{align*}
J_2
&=
\int_\M
\Bigl(S_1(\lambda_1 \alpha) - U_1(\lambda_1 \alpha)\Bigr)
T_2(\lambda_2 \alpha)
T_2(\lambda_3 \alpha)
\prod_{i=1}^{s} G(\mu_i\alpha)
e(\varpi\, \alpha)
K(\alpha, \eta) \dx \alpha 
\\
&
+
\odip{\bm{\lambda},M}
{\eta^2X L^{s}}.
\end{align*}
Using \eqref{dissect-def},  the Cauchy-Schwarz inequality,
Lemmas \ref{BCP-Gallagher}-\ref{Saffari-Vaughan}
with  $Y=P/X$, 
 the  trivial estimates
$\vert G(\mu_i \alpha)\vert \leq L$, $K(\alpha,\eta)\ll \eta^2$,
we have that
\begin{align*}
J_2
&
\ll
\eta^2 L^s 
\Bigl(
\int_{\M}
\vert
S_{1}(\lambda_1\alpha)
-
U_1(\lambda_1\alpha)
\vert^2
\dx \alpha
\Bigr)^{1/2}
\Bigl(
\int_{\M}
\vert
T_{2}(\lambda_2\alpha)
T_{2}(\lambda_3\alpha)
\vert^2
\dx \alpha
\Bigr)^{1/2} 
\\
&
\hskip1cm +
\odip{\bm{\lambda}, M}
{\eta^2XL^{s}} 
\\
&
\ll_{\bm{\lambda},M,\epsilon}
\eta^2 X L^{s}
\exp
\Big(
- \frac{c_{6}(\epsilon)}{2}
\Big(
\frac{\log X}{\log \log X}
\Big)^{1/3}
\Big)
+
\odip{\bm{\lambda}, M}
{\eta^2XL^{s}}
\\
&
=
\odi{\eta^2XL^s},
\end{align*}
since, by \eqref{T2-def-estim},
$\int_{\M}
\vert
T_{2}(\lambda_2\alpha)
T_{2}(\lambda_3\alpha)
\vert^2
\dx \alpha
\ll_{\bm{\lambda}} X$.
Hence  \eqref{J2-estim} holds.

\paragraph{\textbf{Estimation of $J_1$.}}
Recalling that $P= X^{2/5}/\log X$,
using  \eqref{dissect-def} and \eqref{T1-def-estim}-\eqref{I-splitting}
we obtain
\begin{equation}
\label{J1J-relation}
J_1
=
\sum_{1\leq m_1 \leq L} \dotsm \sum_{1\leq m_s \leq L}
\mathcal{J}
\Bigl(
\mu_1 2^{m_1} + \dotsm + \mu_s 2^{m_s}
+ \varpi, \eta
\Bigr)
+
\Odip{\epsilon}{\eta^2  X^{1/5} L^{s+2}},
\end{equation}
where $\mathcal{J}(u, \eta)$ is defined by
\begin{align*}
\mathcal{J}&(u, \eta)
:=
\int_\R
T_1(\lambda_1 \alpha)
T_2(\lambda_2 \alpha)
T_2(\lambda_3 \alpha)
e(u \alpha)
K(\alpha, \eta) \dx \alpha
\\
&=
\frac{1}{4}
  \int_{\epsilon X}^{X}
  \int_{\epsilon X}^{X}
    \int_{\epsilon X}^{X}
	\widehat{K}(\lambda_1u_1+\lambda_2u_2 +\lambda_3u_3+ u , \eta)
	u_2^{-1/2}u_3^{-1/2}
\dx u_1 \dx u_2 \dx u_3
\end{align*}
and the second relation follows by \eqref{T1-def-estim}-\eqref{T2-def-estim} 
and interchanging the integration order.
We recall that $\lambda_{1}<0$,  $\lambda_{2},\lambda_{3}>0$.
If $|u| \le \epsilon X$, for
\[
\frac{X \vert \lambda_1 \vert}
{2
\left(
\vert \lambda_1 \vert + \lambda_2 + \lambda_3 
\right)
}
\leq
u_2,u_3
\leq
\frac{X \vert \lambda_1 \vert}
{
\vert \lambda_1 \vert+ \lambda_2 + \lambda_3
},
\]
sufficiently large $X$ and sufficiently small $\epsilon$, we get that
\[
-\frac{\eta}{2}
-(\lambda_2u_2 +\lambda_3u_3+ u)
\leq
\vert \lambda_1 \vert u_1
\leq
\frac{\eta}{2}
-(\lambda_2u_2 +\lambda_3u_3+ u).
\]
Hence there exists an interval for $u_1$ of length $\eta \vert \lambda_1\vert^{-1}$ and contained in $[\epsilon X, X]$ such that
$\widehat{K}(\lambda_1u_1+\lambda_2u_2 +\lambda_3u_3+ u , \eta) \geq \eta/2$.
So, letting $b=(X\vert \lambda_1 \vert)/(\vert \lambda_1 \vert +  \lambda_2 + \lambda_3)$, 
we can write that
\begin{align*}
\mathcal{J}(u, \eta)
&
\geq
\frac{\eta^2}{8\vert \lambda_1 \vert}
\Bigl(
\int_{b/2}^{b} v^{-1/2} \dx v
\Bigr)^{2}
=
\frac{(3-2\sqrt{2})\eta^2  X}
{4
\left(
\vert \lambda_1 \vert + \lambda_2 + \lambda_3 
\right)
}.
\end{align*}
By the definition of $L$, we have that 
$
\vert
\mu_1 2^{m_1} + \dotsm + \mu_s 2^{m_s}
+ \varpi
\vert
\leq 
\epsilon X$
for $X$ sufficiently large.
Hence by \eqref{J1J-relation} we obtain
\[
J_1
\geq 
\frac{(3-2\sqrt{2})\eta^2X L^s}
{4
\left(
\vert \lambda_1 \vert + \lambda_2 + \lambda_3 
\right)
}
+\Odip{\epsilon}{\eta^2  X^{1/5} L^{s+2}},
\]
thus proving \eqref{J1-lower-bound}.
Arguing analogously we can prove the case
$\lambda_{1},\lambda_{2}<0$,  $\lambda_{3}>0$.

\section{The trivial arc} 
Recalling \eqref{dissect-def}, the trivial estimate 
$\vert G(\mu_i \alpha)\vert \leq L$
and using twice the Cauchy-Schwarz inequality, we get
\begin{align*}
\vert I(X ; \t) \vert
&\ll
L^s
\Bigl(
\int_{L^2}^{+\infty}
\vert
S_1(\lambda_1\alpha)
\vert^2
K(\alpha,\eta)
\dx \alpha
\Bigr)^{1/2}
\\
&
\times
\Bigl(
\int_{L^2}^{+\infty}
\vert
S_2(\lambda_2\alpha)
\vert^4
K(\alpha,\eta)
\dx \alpha
\Bigr)^{1/4}
\Bigl(
\int_{L^2}^{+\infty}
\vert
S_2(\lambda_3\alpha)
\vert^4
K(\alpha,\eta)
\dx \alpha
\Bigr)^{1/4}.
\end{align*}
By \eqref{K-inequality}
and making a change of variable, we have,
for $i=2,3$, that
\begin{align*}
\int_{L^2}^{+\infty}
&\vert
S_2(\lambda_i\alpha)
\vert^4
K(\alpha,\eta)
\dx \alpha 
\ll_{\bm{\lambda}}
\int_{\lambda_i L^2}^{+\infty}
\frac
{\vert
S_2(\alpha)
\vert^4
}
{\alpha^2}
\dx \alpha 
\\
&\ll
\sum_{n\geq \lambda_i L^2}
\frac{1}{(n-1)^2}
\int_{n-1}^{n}
\vert
S_2(\alpha)
\vert^4
\dx \alpha
\ll_{\bm{\lambda}}
L^{-2}
\int_0^{1}
\vert
S_2(\alpha)
\vert^4
\dx \alpha
\ll_{\bm{\lambda},M,\epsilon}
X,
\end{align*}
by \eqref{S2-fourth}.
Moreover, arguing analogously,
\begin{align*}
\int_{L^2}^{+\infty}
&\vert
S_1(\lambda_1\alpha)
\vert^2
K(\alpha,\eta)
\dx \alpha
\ll_{\bm{\lambda}}
\int_{\vert \lambda_1\vert L^2}^{+\infty}
\frac
{\vert
S_1(\alpha)
\vert^2
}
{\alpha^2}
\dx \alpha 
\\
&
\hskip-0.75cm
\ll
\sum_{n\geq \vert \lambda_1\vert L^2}
\frac{1}{(n-1)^2}
\int_{n-1}^{n}
\vert
S_1(\alpha)
\vert^2
\dx \alpha
\ll_{\bm{\lambda}}
L^{-2}
\int_0^{1}
\vert
S_1(\alpha)
\vert^2
\dx \alpha
\ll_{\bm{\lambda},M,\epsilon}
\frac{X}{\log X},
\end{align*}
by \eqref{S1-squared}.
Hence \eqref{trivial-goal} holds.

\section{The minor arc} 
Recalling first
\[
I(X ; \m)
=
\int_{\m}
S_1(\lambda_1 \alpha) S_2(\lambda_2 \alpha)S_2(\lambda_3 \alpha)
\prod_{i=1}^{s} G(\mu_i\alpha)
e(\varpi\, \alpha)
K(\alpha,\eta)
\dx \alpha,
\]
and letting $c\in (0,1)$ to be chosen later,
we first split $\m$ as  $\m_1 \sqcup \m_2$,  
where $\m_2$ is the set of $\alpha\in \m$
such that $\vert G(\mu_{i}\alpha)\vert > \nu(c) L$
for some $i\in \{1,\dotsc,s\}$,
and $\nu(c)$ is defined in Lemma  \ref{minor-arc-power-of-two-estim}.
We will choose $c$ to get
$\vert I(X ; \m_2) \vert = \odi{\eta X}$,
since, again by  Lemma  \ref{minor-arc-power-of-two-estim},
we know that $\vert \m_2 \vert \ll_{M,\epsilon} s L^2 X^{-c}$.

To this end, we first use the trivial estimates
$\vert G(\mu_i \alpha)\vert \leq L$ and $K(\alpha,\eta)\ll \eta^2$ and
Lemma \ref{minor-arc-lemma}
(assuming, without any loss of generality, that $V(\alpha)=\vert S_{2}(\lambda_{2}\alpha) \vert$).
Then, using twice the Cauchy-Schwarz inequality and \eqref{S1-squared}-\eqref{S2-fourth}, we get
\begin{align*} 
\vert I(X ;& \m_2) \vert
 \leq
\eta^{2}L^{s}
\Bigl(
\sup_{\alpha \in \m}
\vert
V(\alpha)
\vert
\Bigr)
\Bigl(
\int_{\m_2}
\vert 
 S_1(\lambda_1 \alpha)
 S_{2}(\lambda_{3}\alpha) 
 \vert 
\dx \alpha
\Bigr)
\\
&
\ll
\eta^{2}L^{s}
X^{7/16+\epsilon}
\vert \m_2 \vert^{1/4}
\Bigl(
\int_{\m_2}
\vert S_1(\lambda_1 \alpha) \vert ^2 
\dx \alpha
\Bigr)^{1/2}
\Bigl(
\int_{\m_2}
\vert S_2(\lambda_3 \alpha) \vert ^4 
\dx \alpha
\Bigr)^{1/4}
\\
&
\ll_{{\bm \lambda}}
\eta^{2}L^{s}
X^{7/16+\epsilon}
\vert \m_2 \vert^{1/4}
\Bigl(
L^2
\int_{0}^1
\vert S_1(\alpha) \vert ^2 
\dx \alpha
\Bigr)^{1/2}
\Bigl(
L^2
\int_{0}^1\vert S_2(\alpha) \vert ^4 
\dx \alpha
\Bigr)^{1/4}
\\
&
\ll_{{\bm \lambda},M, \epsilon}
s^{1/4}\eta^{2}L^{s+3} 
X^{19/16+\epsilon-c/4},
\end{align*} 
where $X=q^{2}$ and $q$ is the denominator of a convergent of the 
continued fraction for $\lambda_{2}/\lambda_{3}$.
 Taking $c=3/4+10^{-20}$ and using \eqref{G-irrational}, we get,
for $\nu= 0.8844472132$ and a
sufficiently small $\epsilon>0$, that
\begin{equation}
\label{minor-arc-2-final}
\vert I(X ; \m_2) \vert
=
\odi{\eta X}.
\end{equation}
We remark that neither the result of Kumchev \cite{Kumchev2006} nor 
the approach of  Cook, Fox and Harman, see \cite{CookF2001}, \cite{CookH2006}, \cite{Harman2004}, seem to give any 
improvement of the previous estimates.

Now we evaluate the contribution of $\m_1$.
Using the  Cauchy-Schwarz inequality,
Lemmas \ref{Dioph-equation} and \ref{sqroot-Dioph-equation},
we have
\begin{align} 
\notag
\vert I(X ; \m_1) \vert
&\leq
(\nu L)^{s-3}
\Bigl(
\int_{\m}
\vert S_1(\lambda_1 \alpha) G(\mu_1 \alpha) \vert ^2
K(\alpha,\eta)
\dx \alpha
\Bigr)^{1/2}
\\
& \times
\Bigl(
\int_{\m}
\vert S_2(\lambda_2 \alpha) G(\mu_2 \alpha) \vert ^4
K(\alpha,\eta)
\dx \alpha
\Bigr)^{1/4}
\\
& \times
\notag
\Bigl(
\int_{\m}
\vert S_2(\lambda_3 \alpha) G(\mu_3\alpha) \vert ^4
K(\alpha,\eta)
\dx \alpha
\Bigr)^{1/4}
\notag
\\
&
<
\nu^{s-3} 
C(q_1,q_2,q_{3}, \epsilon)
 \eta X L^{s},
\label{minor-arc-1}
\end{align}
where $C(q_1,q_2,q_{3}, \epsilon)$ is defined as we did in \eqref{Cq1q2q3-def}.

Hence, by \eqref{minor-arc-2-final}-\eqref{minor-arc-1},
for $X$ sufficiently large
we finally get
\begin{equation} 
\notag
\vert I(X ; \m) \vert
<
(0.8844472132)^{s-3}
C(q_1,q_2,q_{3}, \epsilon)
 \eta X L^{s}.
\end{equation} 

\noindent
This means that \eqref{minor-goal} holds with
$c_2(s) $ $=(0.8844472132)^{s-3}$ $
C(q_1,q_2,q_{3}, \epsilon)$.

\section{Proof of the Theorem} 

We have to verify if there exists an $s_0\in\N$ such that
\eqref{constant-condition} holds for $X$ sufficiently large,
where $X=q^{2}$ and $q$ is the 
denominator of a convergent of the 
continued fraction for $\lambda_{2}/\lambda_{3}$.
Combining the inequalities \eqref{major-goal}-\eqref{minor-goal},
where
$c_2(s)= (0.8844472132)^{s-3}C(q_1,q_2,q_{3}, \epsilon)$ 
we obtain for $s\geq s_0$, $s_0$ defined 
in \eqref{s0-def}, that \eqref{constant-condition} holds.

\subsection*{Acknowledgement} This research was partially supported by the PRIN 2008 grant LMSMTY\_005. We would like to thank the anonymous referee for his/her suggestions.

\end{document}